\documentclass[reqno]{amsart}
\usepackage{amsmath,amssymb,amsthm}
\usepackage{bm,bbm,comment, mathtools}
\usepackage{graphicx}
\usepackage{appendix}
\usepackage{geometry}
\allowdisplaybreaks

\makeatletter
\@namedef{subjclassname@2020}{%
  \textup{2020} Mathematics Subject Classification}
\makeatother
\thanks{This work was supported by JST SPRING, Grant Number JPMJSP2110. }
\author[N. Takeda]{Nobuki TAKEDA}
\address{Department of Mathematics, Graduate School of Science, Kyoto University, Kyoto 606-8502, Japan}
\email{takeda.nobuki.72z@st.kyoto-u.ac.jp}

\theoremstyle{definition}
\newtheorem{dfn}{Definition}[section]
\newtheorem{rem}[dfn]{Remark}
\newtheorem{ex}[dfn]{Example}
\theoremstyle{plain}
\newtheorem{prop}[dfn]{Proposition}
\newtheorem{dfnprop}[dfn]{Definition-Proposition}

\newtheorem*{cond}{Condition (A)}
\newtheorem*{claim}{Claim}
\newtheorem{lem}[dfn]{Lemma}
\newtheorem{thm}[dfn]{Theorem}
\newtheorem{cor}[dfn]{Corollary}

\newcommand{\bbc}{\mathbb{C}}
\newcommand{\bbr}{\mathbb{R}}
\newcommand{\bbq}{\mathbb{Q}}
\newcommand{\bbz}{\mathbb{Z}}
\newcommand{\rmu}{\mathrm{U}}
\newcommand{\bfa}{\mathbf{a}}
\newcommand{\bfb}{\mathbf{b}}
\newcommand{\bfh}{\mathbf{h}}
\newcommand{\bfn}{\mathbf{n}}
\newcommand{\bfe}{\mathbf{E}}
\newcommand{\caln}{\mathcal{N}}
\newcommand{\calp}{\mathcal{P}}
\newcommand{\calk}{\mathcal{K}}
\newcommand{\calm}{\mathcal{M}}
\newcommand{\call}{\mathcal{L}}
\newcommand{\frakn}{\mathfrak{n}}
\newcommand{\bfi}{\boldsymbol{i}}
\newcommand{\adele}{\mathbb{A}}
\newcommand{\gn}{\mathrm{G}_\bfn}
\newcommand{\kp}{K^+}
\newcommand{\dkl}{\mathbb{D}_{\boldsymbol{k}, \boldsymbol{l}}}
\renewcommand{\Re}{\mathrm{Re}}
\renewcommand{\Im}{\mathrm{Im}}
\newcommand{\bfk}{\boldsymbol{k}}
\newcommand{\bfl}{\boldsymbol{l}}
\newcommand{\kv}{\boldsymbol{k}_v}
\newcommand{\lv}{\boldsymbol{l}_v}
\newcommand{\va}{v\in\mathbf{a}}

\DeclareMathOperator*{\bigboxtimes}{ \boxtimes}

\providecommand{\abs}[1]{\left\lvert#1\right\rvert}
\providecommand{\norm}[1]{\left\lvert#1\right\rvert}
\providecommand{\adj}[1]{{#1^*}}
\providecommand{\uni}[1]{\mathrm{U}_{#1}}
\providecommand{\hus}[1]{\mathfrak{H}_{#1}}
\providecommand{\inte}[1]{\mathcal{O}_{#1}}
\providecommand{\herm}[1]{\mathbb{S}_{#1}}
\numberwithin{equation}{section}
\DeclareMathOperator{\tr}{Tr}

\begin{document}
\title{Differential operators on hermitian modular forms on $\rmu(n,n)$}
\date{\today}
\maketitle
\begin{abstract}
	We construct explicit differential operators on hermitian modular forms, extending methods developed for Siegel modular forms.
	These differential operators are closely related to the two-variable spherical pluriharmonic polynomials.
	We construct explicit bases for the space of such polynomials and use them to build concrete operators.
	As an application, we derive an exact pullback formula for hermitian Eisenstein series.
\end{abstract}
\tableofcontents

\section{Introduction}
In the theory of Siegel modular forms, differential operators have been explicitly constructed and employed by various researchers to investigate both algebraic and analytic properties, including Fourier expansions, inner products, and their relations to special values of $L$-functions (see for example \cite{Bocherer1985Uber}, \cite{BoSaYa1992pullback}, \cite{Kozima2021Pullback}, \cite{Katsurada2010exact}).
In particular, Ibukiyama proposed a general framework for differential operators on Siegel modular forms, introducing conditions under which these operators preserve automorphic properties \cite{ibukiyama1999differential}.
This framework has since been further developed using representation-theoretic methods, shedding light on deeper structural aspects \cite{Takeda2025Kurokawa}.

In contrast, the theory of differential operators for Hermitian modular forms remains relatively undeveloped, with only a handful of known results in specific cases \cite{Browning2024Constructing}.
The aim of this paper is to generalize the techniques developed in the context of Siegel modular forms to the Hermitian setting. Specifically, we construct explicit differential operators on Hermitian modular forms and investigate their structural and arithmetic applications.
\medskip

We now formulate a specific condition satisfied by these differential operators.

Let $K$ be a quadratic imaginary extension of a totally real field $\kp$, and let $\bfa$ (resp.\ $\bfh$) denote the set of infinite (resp.\ finite) places of $\kp$.
We define $m:=m_{\kp}:=\#\bfa = [\kp:\bbq]$.
We write $\hus{n}^{\bfa}$ for the product of $m$ copies of the Hermitian upper half-space $\hus{n}$.
Let $M_\rho(\Gamma_K^{(n)}(\frakn))$ denote the complex vector space of Hermitian modular forms of weight $\rho$ and level $\frakn$ on $\hus{n}^\bfa$.
For a dominant integral weight $\bfk$ of length at most $n$, $(\rho_{n, \bfk}, V_{n,\bfk})$ denote the corresponding irreducible representation of $\mathrm{GL}_n(\bbc)$ with highest weight $\bfk$.

Let $(\bfk,\bfl) = (\kv,\lv)_{\va}$ be a family of pairs of dominant integral weights, with $\ell(\kv)\leq n$ and $\ell(\lv)\leq n$ for each $v \in \bfa$.
We then define the representation $\rho_{n,(\bfk,\bfl)} = \bigboxtimes_{\va} (\rho_{n, \kv} \boxtimes \rho_{n, \lv})$ of the group $K_{n,\infty}^\bbc := \prod_{\va} (\mathrm{GL}_n(\bbc) \times \mathrm{GL}_n(\bbc))$.
Let $G_n$ denote a unitary group defined over $\kp$.

Fix integers $n_1,\ldots,n_d$ with $n_1 \geq \cdots \geq n_d \geq 1$ and set $n = n_1 + \cdots + n_d$.
Let $(\rho_s, V_s)$ be a representation of $K_{n_s,\infty}^\bbc$ for $s=1,\ldots,d$
and let $\kappa = (\kappa_v)_{\va}$ be a family of positive integers.
We consider differential operators $\mathbb{D}$ valued in $V := V_{1} \otimes \cdots \otimes V_{d}$, which act on scalar-valued functions on $\hus{n}$ and satisfy the following condition:

\begin{cond}
	For any modular form $F \in M_\kappa(\Gamma_K^{(n)})$, we have
	\[
		\mathrm{Res}(\mathbb{D}(F)) \in \bigotimes_{i=1}^d M_{{\det}^\kappa \rho_{n_i}}(\Gamma_K^{(n_i)}),
	\]
	where $\mathrm{Res}$ denotes the restriction of a function on $\hus{n}^\bfa$ to $\hus{n_1}^\bfa \times \cdots \times \hus{n_d}^\bfa$.
\end{cond}

This condition corresponds to Case (I) in \cite{ibukiyama1999differential}.

We write $\partial_Z := \left( \frac{\partial}{\partial Z_{v,ij}} \right)_{\va}$.
Let $P_v(X)$ be a vector-valued polynomial on the space $M_n$ of $n \times n$ matrices.
The following proposition characterizes when the operator $\mathbb{D} = P(\partial_Z) := (P_v(\partial_{Z_v}))_{\va}$ satisfies Condition (A):

\begin{prop}[Proposition~\ref{prop:diff}]
	Let $n_1,\ldots,n_d$ be positive integers with $n_1 \geq \cdots \geq n_d \geq 1$ and $n = n_1 + \cdots + n_d$.
	Fix a family $(\bfk_s, \bfl_s) = (\bfk_{s,v}, \bfl_{s,v})_{\va}$ of dominant integral weights satisfying $\ell(\bfk_{s,v}) \leq n_d$, $\ell(\bfl_{s,v}) \leq n_d$, and $\ell(\bfk_{s,v}) + \ell(\bfl_{s,v}) \leq \kappa_v$ for each $v \in \bfa$ and $s=1,\ldots,d$.

	Let $P_v(T)$ be a polynomial valued in $V_{n_1,\bfk_{v,1},\bfl_{v,1}} \otimes \cdots \otimes V_{n_d,\bfk_{v,d},\bfl_{v,d}}$ on $M_n$.
	Then $\mathbb{D} = \prod_{\va} P_v(\partial_{Z_v})$ satisfies Condition (A) for weight ${\det}^\kappa$ and representation ${\det}^\kappa \rho_{n_1,\bfk_1,\bfl_1} \otimes \cdots \otimes \rho_{n_d,\bfk_d,\bfl_d}$ if and only if each $P_v(T)$ satisfies the following:

	\begin{enumerate}
		\item Let $\widetilde{P_v}(X_1,\ldots,X_d,Y_1,\ldots,Y_d)$ be defined by
		      \[
			      \widetilde{P_v}(X_1,\ldots,X_d,Y_1,\ldots,Y_d) := P_v\left(
			      \begin{pmatrix}
					      X_1\,^tY_1 & \cdots & X_1\,^tY_d \\
					      \vdots     & \ddots & \vdots     \\
					      X_d\,^tY_1 & \cdots & X_d\,^tY_d
				      \end{pmatrix}
			      \right),
		      \]
		      where $X_i, Y_i \in M_{n_i, \kappa_v}$.
		      Then $\widetilde{P_v}$ is pluriharmonic, i.e.,
		      \[\sum_{s=1}^\kappa \frac{\partial^2\widetilde{P_v}}{\partial X_{i,s}\partial Y_{j,s}}=0\]
		      in each pair $(X_i, Y_i)$.
		\item For $(A_i, B_i) \in K_{n_i,v}^\bbc := \mathrm{GL}_{n_i}(\bbc) \times \mathrm{GL}_{n_i}(\bbc)$, we have
		      \[
			      P_v\left(
			      \begin{pmatrix} A_1 & & \\ & \ddots & \\ & & A_d \end{pmatrix}
			      T
			      \begin{pmatrix} ^tB_1 & & \\ & \ddots & \\ & & ^tB_d \end{pmatrix}
			      \right)
			      =
			      \left( \bigotimes_{i=1}^d \rho_{n_i,\bfk_{i,v},\bfl_{i,v}}(A_i,B_i) \right) P_v(T).
		      \]
	\end{enumerate}
\end{prop}

Condition (1) implies that these differential operators are closely related to two-variable spherical pluriharmonic polynomials.
Accordingly, a substantial part of this paper is devoted to analyzing the space of such higher-degree pluriharmonic polynomials, building on the work of Ibukiyama and Zagier \cite{Ibukiyama2014Higher}, and to providing explicit descriptions of their bases, namely the ``descending basis'' and the ``monomial basis''.

These bases are used to construct explicit differential operators.
In particular, for weights corresponding to symmetric power representations, we obtain a fully explicit expression (Theorem~\ref{thm:diffsym}).

Using these constructions, we derive an explicit formula for the coefficients appearing in the pullback formula for Hermitian Eisenstein series, as formulated in \cite{Takeda2025pullback} (see Theorem~\ref{thm:c}).
The pullback formula itself follows from the double coset decompositions and the automorphic properties of the differential operators; the main point addressed here is the explicit evaluation of the archimedean gamma factor.
Such formulas are expected to be useful in arithmetic applications, for example to critical values of standard $L$-functions and to Fourier expansions of Hermitian Klingen-Eisenstein series.

This paper is organized as follows:
Section~2 introduces Hermitian modular forms and the notation used throughout the paper.
Section~3 reviews differential operators on Hermitian modular forms and the conditions they satisfy.
Section~4 studies the space of higher spherical pluriharmonic polynomials and provides explicit bases.
Section~5 constructs concrete differential operators based on these results.
Finally, Section~6 applies these operators to derive the exact pullback formula.

\vskip.5\baselineskip
\paragraph{\textbf{Acknowledgment}}
The author would like to thank T. Ikeda for his great guidance and support as my supervisor,
and H. Katsurada for his many suggestions and comments.
The author also thanks the anonymous reviewer for constructive comments.
This work was supported by JST SPRING, Grant Number JPMJSP2110.
\vskip.5\baselineskip
\paragraph{\textbf{Notation}}
We denote by $M_{m,n}(R)$ the set of $m\times n$ matrices with entries in $R$.
In particular, we put $M_n(R) := M_{n,n}(R)$.
Let $I_n$ be the identity element of $M_n(R)$ and $e_{ij}$
the matrix with $1$ at the $(i,j)$-th entry and $0$ elsewhere.
Let $\det(X)$ be the determinant of $X$ and $\tr(X)$ the trace of $X$,
${}^tX$ the transpose of $X$ for a square matrix $x \in M_n(R)$.
Let $\mathrm{GL}_n(R) \subset M_n(R)$ be a general linear group of degree $n$.

Let $K$ be a quadratic extension field of $K_0$ with the non-trivial automorphism $\rho$ of $K$ over $K_0$,
and we often put $\overline{x}=\rho(x)$ for $k\in K$.
We put $\overline{X}=(\overline{x_{ij}})$ and $\adj{X}= {}^t\overline{X}$ for $X=(x_{ij})\in M_{m,n}(K)$.
Let $\adj{B}$ be the transpose of $\overline{B}$ and $\overline{B}$ the complex conjugate of $B$.

Let $K$ be an algebraic number field, and let $\mathfrak{p}$ be a prime ideal of $K$.
We denote by $K_\mathfrak{p}$ a $\mathfrak{p}$-adic completion of $K$
and by $\inte{K}$ the integer ring of $K$.

Let $\herm{n}\subset M_n(\bbc)$ be the set of Hermitian matrices. For an element $X \in \herm{n}$,
we write $X>0$ (resp.\ $X \geq 0$) if $X$ is a positive definite matrix (resp.\ a non-negative definite matrix).
For a subset $S \subset \herm{n}$, we denote by $S_{>0}$ (resp.\ $S_{\geq 0}$)
the subset of positive definite (resp.\ non-negative definite) matrices in $S$.

If a group $G$ acts on a set $V$ then, we denote by $V^G$ the $G$-invariant
subspace of $V$.

Let ${\det}^k$ be the one-dimensional representation of $\mathrm{GL}_n(\bbc)$ given by the $k$-th power of the determinant,
and let $Sym^l$ be the $l$-th symmetric power representation of $\mathrm{GL}_n(\bbc)$.
For a representation $(\rho,V)$, we denote by $(\rho^*, V^*)$ the
contragredient representation of $(\rho, V)$.

We denote by
$(x)^{(r)} = x(x+1)\cdots(x+r-1)$ (resp.\ $(x)_{(r)} = x(x-1)\cdots(x-r+1)$)
the ascending (resp.\ descending) Pochhammer symbol.

\section{Hermitian modular forms}
Let $K$ be a quadratic imaginary extension of a totally real field $F$, and let $\bfa$ (resp.\ $\bfh$) denote the set of infinite (resp.\ finite) places of $F$.
We fix a CM type $\Sigma_K$ for $K$ (i.e. $\Sigma_K$ contains a choice of exactly one representative from each pair of complex conjugate embeddings of $K$).
We often canonically identify CM type $\Sigma_K$ with $\bfa$, and we denote by $\sigma_v$ the embedding $K \hookrightarrow \bbc$ corresponding to $v \in \bfa$.

We put $m=m_{\kp}:=\#\bfa=[\kp:\bbq]$.
We put $K_v=\prod_{w|v} K_w$ and $\mathcal{O}_{K_v}=\prod_{w|v}\mathcal{O}_{K_w}$ for a place $v$ of $\kp$.
Let $\adele=\adele_{\kp}$ be the adele ring of $\kp$, and $\adele_{0}$,
$\adele_{\infty}$ the finite and infinite parts of $\adele$, respectively.

We put $J_n=\begin{pmatrix}O_n & I_n \\ -I_n & O_n \\ \end{pmatrix}$.
The unitary group $\uni{n}$ is an algebraic group defined over $\kp$,
whose $R$-points are given by
\[\uni{n}(R)=\{g\in \mathrm{GL}_{2n}(K\otimes_{\kp}R) \mid \adj{g}J_ng=J_n\}\]
for each $\kp$-algebra $R$.\\ We also define other unitary groups $\rmu(n,n)$
and $\rmu(n)$ by
\begin{align*}
	\rmu(n,n) & =\{g\in \mathrm{GL}_{2n}(\bbc) \mid \adj{g}J_ng=J_n\}, \\
	\rmu(n)   & =\{g\in \mathrm{GL}_{n}(\bbc) \mid \adj{g}g =I_n\}.
\end{align*}
Put $G_n=\rmu_n(\kp)$, $G_{n,v}=\rmu_n(\kp_v)$  for a place $v$ of $\kp$,
$G_{n,\adele}=\rmu_n(\adele)$, $G_{n,0}=\rmu_n(\adele_0)$,
and $G_{n,\infty}=\prod_{\va}G_{n,v}=\prod_{\va}\rmu(n,n)$.

We define $K_{n,v}$ by
\[K_{n,v}=\left\{
	\begin{array}{ll}
		\uni{n}(\inte{\kp_v}) & (v\in\bfh), \\
		\rmu(n)\times \rmu(n) & (\va).
	\end{array}
	\right.\]
Then $K_{n,v}$ is isomorphic to a maximal compact subgroup of $G_{n,v}$.
We fix a maximal compact subgroup of $G_{n,v}$,
which is also denoted by $K_{n,v}$ by abuse of notation.
We put $K_{n,0}=\prod_{v\in\bfh}K_{n,v}$ and $K_{n,\infty}=\prod_{\va }K_{n,v}$.

\subsection{As analytic functions on Hermitian symmetric spaces}
We have the identification
\begin{align*}
	M_n(\bbc) & \cong \herm{n}\otimes_\bbr\bbc  \\
	Z         & \mapsto \Re(Z)+\sqrt{-1}\Im(Z),
\end{align*}
with the Hermitian real part $\Re(Z)$ and the imaginary part $\Im(Z)$, i.e.,
\begin{align*}
	\Re(Z) & =\frac{1}{2}(Z+\adj{Z}),          \\
	\Im(Z) & =\frac{1}{2\sqrt{-1}}(Z-\adj{Z}).
\end{align*}

Let $\hus{n}$ be the Hermitian upper half space of degree $n$, that is,
\[ \hus{n} = \left\{Z \in M_n(\bbc) \mid \Im(Z) >0\right\}.\]

Then $G_{n,\infty}=\prod_{\va}\rmu(n,n)$ acts on $\hus{n}^\bfa$ by
\[g\left<Z\right>=\left((A_vZ_v+B_v)(C_vZ_v+D_v)^{-1}\right)_{\va}\]
for $g=\begin{pmatrix}A_v& B_v \\C_v & D_v \\\end{pmatrix}_{\va} \in G_{n,\infty}$
and $Z=(Z_v)_{\va}\in\hus{n}^\bfa$.
We put $\bfi_n:=(\sqrt{-1}I_n)_{\va}\in\hus{n}^\bfa$.

Let $(\rho, V )$ be an algebraic representation of
$K_{n,\infty}^\bbc:=\prod_{\va}(\mathrm{GL}_n(\bbc)\times\mathrm{GL}_n(\bbc))$ on a finite dimensional complex vector space $V$,
and take a Hermitian inner product on $V$ such that
\[\left<\rho(g)v,w\right>=\left<v,\rho(\adj{g})w\right> \]
for any $g \in K_{n,\infty}^\bbc$.

For  $g=\begin{pmatrix}A_v& B_v \\C_v & D_v \\\end{pmatrix}_{\va} \in G_{n,\infty}$
and $Z=(Z_v)_{\va}\in\hus{n}^\bfa$,
we put
\[\lambda(g,Z)=(C_vZ_v+D_v)_{\va}, \quad \mu(g,Z)=(\overline{C_v}{}^tZ_v+\overline{D_v})_{\va},\]
and
\[ M(g,Z)=(\lambda(g,Z),\mu(g,Z)).\]
We write
\[\lambda(g)=\lambda(g,\bfi_n),\quad \mu(g)=\mu(g,\bfi_n)\ \text{ and }\ M(g)=M(g,\bfi_n)\]
for short.
For a $V$-valued function $F$ on $\hus{n}^\bfa$, we put
\[ F|_{\rho}[g](Z)=\rho(M(g,Z))^{-1}F(g\left< Z\right>) \quad (g\in G_{n,\infty},\ Z \in \hus{n}^\bfa).\]

We put
\[\Gamma_K^{(n)}(\frakn)=\left\{\left.g\in G_{n,\infty}\cap \mathrm{GL}_{2n}(\inte{K})\right| g\equiv I_{2n} \mod{\frakn \mathcal{O}_K}\right\}\]
for an integral ideal $\frakn$ of $\kp$.
When $\frakn=\inte{\kp}$, we put $\Gamma_K^{(n)}=\Gamma_K^{(n)}(\inte{\kp})$.

\begin{dfn}
	We say that $F$ is a (holomorphic) Hermitian modular form of level $\frakn$, and weight $(\rho,V)$
	if $F$ is a holomorphic $V$-valued function on $\hus{n}^\bfa$ and $F|_{\rho}[g]=F$ for all $g \in \Gamma_K^{(n)}(\frakn)$.
	(If $n=1$ and $\kp=\bbq$, another holomorphy condition at the cusps is also needed.)

	We denote by $M_\rho(\Gamma_K^{(n)}(\frakn))$ a complex vector space of all Hermitian modular forms
	of level $\frakn$, and weight $(\rho,V)$.
\end{dfn}

We set
\[\Lambda_n(\frakn)=\left\{T\in\herm{n}\cap  M_n(K)\left|\sum_{\sigma\in\Sigma_K}(\tr(T_\sigma \xi_\sigma))\in \bbz \text{ for any } \xi\in \herm{n}\cap M_n(\frakn \mathcal{O}_K)  \right.\right\},\]
and let $\Lambda_n(\frakn)_{\geq0}$ be a subset of $\Lambda_n(\frakn)$ consisting of non-negative definite matrices (for all embeddings $\sigma\in \Sigma_K$).
Then, a modular form $F \in M_\rho(\Gamma_K^{(n)}(\frakn))$ has the Fourier expansion
\[ F(Z)=\sum_{T\in \Lambda_n(\frakn)_{\geq0}}a(F,T)\mathbf{e}\left(\sum_{\va}(\tr(T_{\sigma_v} Z_v))\right),\]
where $a(F,T)\in V$, $\mathbf{e}(z)=\exp(2\pi\sqrt{-1}z)$.
Here, $T_{\sigma_v}$ is the image of $T\in \herm{n}$ by the embedding corresponding to $\sigma_v\in\Sigma_K$.
If $a(F, T)=0$ unless $T$ is positive definite,
we say that $F$ is a  (holomorphic) Hermitian cusp form of level $\frakn$, and weight $(\rho,V)$.
We also denote by $S_\rho(\Gamma_K^{(n)}(\frakn))$ a complex vector space of all cusp forms of level $\frakn$, and weight $(\rho,V)$.

Write the variable $Z=(X_v+\sqrt{-1}Y_v)_{\va}$ on $\hus{n}^\bfa$ with $X_v,Y_v \in \herm{n}$ for each $\va$. We
identify $\herm{n}$ with $\bbr^{n^2}$ and define measures $dX_v, dY_v$ as the
standard measures on $\bbr^{n^2}$. We define a measure $dZ$ on $\hus{n}^\bfa$ by
\[dZ= \prod_{\va}dX_v dY_v.\]
For $F,G \in M_\rho(\Gamma_K^{(n)}(\frakn))$, we can define the Petersson inner product
as
\[ (F,G)=\int_{D } \left<\rho(Y^{1/2},{}^tY^{1/2})F(Z),\rho(Y^{1/2},{}^tY^{1/2})G(Z)\right>(\prod_{\va}\det(Y_v)^{-2n})dZ,\]
where $Y=(Y_v)_{\va}=\Im(Z)$, $Y^{1/2}=(Y^{1/2}_v)_{\va}$ is a family of positive definite Hermitian matrices such that
$(Y_v^{1/2})^2=Y_v$, and $D$ is a Siegel domain on $\hus{n}^\bfa$
for $\Gamma_K^{(n)}(\frakn)$. This integral converges if either $F$ or $G$ is a cusp
form.

We call a sequence of non-negative integers $\bfk=(k_1,k_2,\ldots)$ a dominant integral weight
if $k_i \geq k_{i+1}$ for all $i$, and $k_i=0$ for almost all $i$.
The largest integer $m$ such that $k_m \neq 0$ is called the length of $\bfk$ and denoted by $\ell(\bfk)$.
The set of dominant integral weights with length less than or equal to $n$
corresponds bijectively to the set of irreducible algebraic representations of $\mathrm{GL}_n(\bbc)$.

For a family $(\bfk,\bfl)=(\kv,\lv)_{\va}$ of pairs of dominant integral weights
such that $\ell(\kv)\leq n$ and $\ell(\lv)\leq n$ for any $v\in \bfa$,
we define the representation
$\rho_{n,(\bfk,\bfl)}=\bigboxtimes_{\va}(\rho_{n, \kv}\boxtimes\rho_{n, \lv})$ of $K_{n,\infty}^\bbc$.
We put $M_{(\bfk,\bfl)}(\Gamma_K^{(n)}(\frakn))=M_{\rho_{n,(\bfk,\bfl)}}(\Gamma_K^{(n)}(\frakn))$
and $S_{(\bfk,\bfl)}(\Gamma_K^{(n)}(\frakn))=S_{\rho_{n,(\bfk,\bfl)}}(\Gamma_K^{(n)}(\frakn))$.
When $\bfk=(\kappa_v,\ldots,\kappa_v)_{\va}$ and $\bfl=(0,\ldots,0)_{\va}$
for a family $\kappa=(\kappa_v)_{\va}$ of non-negative integers,
we also put ${\det}^\kappa=\rho_{n,(\bfk,\bfl)}$,
$M_\kappa(\Gamma_K^{(n)}(\frakn))=M_{(\bfk,\bfl)}(\Gamma_K^{(n)}(\frakn))$
and $S_\kappa(\Gamma_K^{(n)}(\frakn))=S_{(\bfk,\bfl)}(\Gamma_K^{(n)}(\frakn))$.

\subsection{As functions on $\rmu(n,n)$}
Let $K_{n,\infty}$ be the stabilizer of $\bfi_n \in \hus{n}^\bfa$ in
$G_{n,\infty}$. Then, $K_{n,\infty}$ is a maximal compact subgroup of $G_{n,\infty}$
and isomorphic to $\prod_{\va}\rmu(n)\times\rmu(n)$, which is given by
\[\begin{array}{rccc}
		 & \prod_{\va}\rmu(n)\times\rmu(n) & \rightarrow & K_{n,\infty}                                                                                               \\
		 & (k_{1,v},k_{2,v})_{\va}         & \mapsto     & \left(\mathfrak{c}\begin{pmatrix}k_{2,v}&0\\0&{}^tk_{1,v}^{-1}\end{pmatrix}\mathfrak{c}^{-1}\right)_{\va},
	\end{array}\]
where $\mathfrak{c}=\dfrac{1}{\sqrt{2}}\begin{pmatrix}1 & \sqrt{-1}\\ \sqrt{-1} & 1\\ \end{pmatrix}\in M_{2n}(\bbc)$.
Here we are taking this slightly strange isomorphism for Proposition~\ref{prop:automisom}.

We put $\mathfrak{g}_{n,v}=\mathrm{Lie}(G_{n,v})$,
$\mathfrak{k}_{n,v}=\mathrm{Lie}(K_{n,v})$ and let $\mathfrak{g}^\bbc_{n,v}$ and
$\mathfrak{k}^\bbc_{n,v}$ be the complexification of $\mathfrak{g}_{n,v}$ and
$\mathfrak{k}_{n,v}$, respectively.
We have the Cartan decomposition $\mathfrak{g}_{n,v}=\mathfrak{k}_{n,v}\oplus\mathfrak{p}_{n,v}$.
Furthermore, we put
\begin{align*}
	\kappa_{v,ij}  & =\mathfrak{c}\begin{pmatrix}0 & 0\\ 0 & -e_{ji}\\ \end{pmatrix}\mathfrak{c}^{-1},\quad
	\kappa'_{v,ij}=\mathfrak{c}\begin{pmatrix}e_{ij} & 0\\ 0 &0 \\ \end{pmatrix}\mathfrak{c}^{-1},                            \\
	\pi^{+}_{v,ij} & =\mathfrak{c}\begin{pmatrix}0 & e_{ij}\\ 0 & 0\\ \end{pmatrix}\mathfrak{c}^{-1},\quad \mathrm{and} \quad
	\pi^{-}_{v,ij}=\mathfrak{c}\begin{pmatrix}0 & 0\\ e_{ij} & 0\\ \end{pmatrix}\mathfrak{c}^{-1},
\end{align*}
where
$e_{ij} \in M_{n,n}(\bbc)$ is the matrix whose only non-zero entry is 1 in $(i,j)$-component.
$\{\kappa_{v,ij}\}$ is a basis of $\mathfrak{k}^\bbc_{n,v}$.
Let $\mathfrak{p}^+_{n,v}$ (resp. $\mathfrak{p}^-_{n,v}$) be the $\bbc$-span of
$\{\pi^{+}_{v,ij}\}$ (resp. $\{\pi^{-}_{v,ij}\}$) in $\mathfrak{g}^\bbc_{n,v}$.
And then, put $\mathfrak{g}_{n}=\prod_{\va}\mathfrak{g}_{n,v}$, $\mathfrak{k}^\bbc_{n}=\prod_{\va}\mathfrak{k}^\bbc_{n,v}$, etc.

For a representation $(\rho, U_\rho)$ of $K_{n,\infty}$, we define the representation $(\rho',U_{\rho'}\ (=U_\rho))$ by
$\rho'(g_1,g_2)=\rho({}^tg_1^{-1},{}^tg_2^{-1})$, which is isomorphic to the contragredient representation $\rho^*$.

\begin{dfn}
	Let $(\rho, U_\rho)$ be an irreducible unitary representation of $K_{n,\infty}$
	and $\Gamma_{n}$ be a discrete subgroup of $G_{n}$.
	We embed $\Gamma_{n}$ diagonally into $G_{n,\infty}$ and consider it as the subgroup of $G_{n,\infty}$.
	Then, a Hermitian modular form of type $\rho$ for $\Gamma_{n}$ is
	a $U_{\rho'}$-valued $C^\infty$-function $\phi$ on $G_{n,\infty}$ that satisfies the following conditions:
	\begin{enumerate}
		\item $\phi(\gamma gk)=\rho'(k)^{-1}\phi(g)$ for $k\in K_{n,\infty}$ and $\gamma\in \Gamma_{n}$,
		\item $\phi$ is annihilated by the right derivation of $\mathfrak{p}_{n}^-$,
		\item $\phi$ is of moderate growth.
	\end{enumerate}
\end{dfn}

We denote the space of moderate growth $C^\infty$-functions on $G_{n,\infty}$
which are invariant under left translation by $\Gamma_{n}$ by $C_\mathrm{mod}^\infty(\Gamma_{n}\backslash G_{n,\infty})$
and the space consisting of all Hermitian modular forms of type $\rho$ for $\Gamma_n$
by $\left[C_\mathrm{mod}^\infty(\Gamma_{n}\backslash G_{n,\infty})\otimes U_{\rho'}\right]^{K_{n,\infty},\mathfrak{p}_{n}^-=0}$.

For $f \in M_{\rho}(\Gamma_K^{(n)})$, we define a $U_{\rho}$-valued $C^\infty$-function $\phi_f$ on $G_{n,\infty}$ by
\[\phi_f(g)=(f|_{\rho} g )(\sqrt{-1})=\rho(M(g))^{-1}f(g\left<\bfi_n\right>)\]
for $g\in G_{n,v}$. Then, we have the following proposition.

\begin{prop}[e.g. \cite{Eischen2024Automorphic}]\label{prop:automisom}
	The above correspondence $f\mapsto \phi_f$ gives the isomorphism
	\[M_\rho(\Gamma_K^{(n)})\overset{\sim}{\longrightarrow}
		\left[C_\mathrm{mod}^\infty(\Gamma_K^{(n)}\backslash G_{n,v})\otimes U_{\rho'}\right]^{K_{n,\infty},\mathfrak{p}_{n}^-=0}.\]
\end{prop}

\subsection{As functions on unitary groups over the adeles}
There is a unique compact open subgroup $K_{n,0}(\frakn)$ of $G_{n,0}$ such that
\[\Gamma_K^{(n)}(\frakn)=G_n\cap K_{n,0}(\frakn)K_{n,\infty}\]
for an integral ideal $\frakn$ of $\kp$.
We put $K_{n,v}(\frakn)=K_{n,0}(\frakn)_v$.
We remark that $K_{n,v}(\frakn)=K_{n,v}$ for a finite place $v\nmid\frakn$.

\begin{dfn}
	A Hermitian automorphic form on $G_{n,\adele}$ of level $\frakn$, and weight $(\rho,V)$ is defined to be a $V$-valued smooth function $f$ on $G_{n,\adele}$
	such that left $G_{n}$-invariant, right $K_{n,0}(\frakn)$-invariant, right-equivariant under the action of $(K_{n,\infty}, \rho)$,
	of moderate growth, and $Z(\mathfrak{g})$-invariant, where $Z(\mathfrak{g})$ denotes the center of the complexified Lie algebra $\mathfrak{g}$ of $\rmu(n,n)$.

	We denote by $\mathcal{A}_n(\rho,\frakn)$ the complex vector space of Hermitian automorphic forms on $G_{n,\adele}$ of weight $\rho$.
\end{dfn}

\begin{dfn}
	A Hermitian automorphic form $f\in\mathcal{A}_n(\rho,\frakn)$ is called a cusp form if
	\[\int_{N(\kp)\backslash N(\adele)}f(ng)dn=0\]
	for any $g \in G_{n,\adele}$ and any unipotent radical $N$ of each proper
	parabolic subgroup of $\rmu_n$.

	We denote by $\mathcal{A}_{0,n}(\rho,\frakn)$ the complex vector space of cusp
	forms on $G_{n,\adele}$ of weight $\rho$.
\end{dfn}

We put
\[ g_Z =\begin{pmatrix}
		Y^{1/2} & XY^{-1/2} \\ 0_n& Y^{-1/2}
	\end{pmatrix}\in G_{n,\infty}\]
for $Z=X+\sqrt{-1}Y\in\hus{n}^\bfa$.
For $f\in \mathcal{A}_n(\rho,\frakn)$, we define a function $\hat{f}$ on
$\hus{n}^\bfa$ by
\[\hat{f}(Z)=\rho(M(g_z))f(g_z).\]
Then, we have $\hat{f}\in M_\rho(\Gamma_K^{(n)}(\frakn))$.
Moreover, if $f \in	\mathcal{A}_{0,n}(\rho,\frakn)$,
then we have $\hat{f}\in S_\rho(\Gamma_K^{(n)}(\frakn))$.

For $v\in \bfh$, we take the Haar measure $dg_v$ on $G_{n,v}$ normalized so that the volume of $K_{n,v}$ is 1.
For $v\in \bfa$, we take the Haar measure $dg_v$ on $G_{n,v}$ such that the volume of $K_{n,v}$ is 1
and the Haar measure on $\hus{n}\cong G_{n,v}/K_{n,v}$ induced from $dg_v$ is $(\det Y_v)^{-2n}dZ_v$.
Using these, we fix the Haar measure $dg=\prod_vdg_v$ on $G_{n,\adele}$.
We define the Petersson inner product on $\mathcal{A}_n(\rho,\frakn)$ as
\[ (f,h)=\int_{G_{n}\backslash G_{n,\adele}} \left<f(g),h(g)\right>dg,\]
for $f,h \in \mathcal{A}_n(\rho,\frakn)$, where $dg$ is a Haar measure on $G_{n}\backslash G_{n,\adele}$
induced from that on $G_{n,\adele}$.

For a finite place $v\in\bfh$ such that corresponds to a prime ideal $\mathfrak{p}$ of $\kp$,
let $\mathcal{H}_{n,\mathfrak{p}}$ be the convolution algebra of
left and right $K_{n,v}$-invariant compactly supported $\bbq$-valued functions of $G_{n,v}$,
which is called the spherical Hecke algebra at $\mathfrak{p}$.
The spherical Hecke algebra $\mathcal{H}_{n,\mathfrak{p}}$ at $\mathfrak{p}$ acts
on the set of continuous right $K_{n,v}$-invariant functions on $G_{n,v}$ (or on $G_{n,\adele}$) by right convolution,
i.e., for a continuous right $K_{n,v}$-invariant function $f$ on $G_{n,v}$ (or on $G_{n,\adele}$) and $\eta\in\mathcal{H}_{n,\mathfrak{p}}$, we put
\[(\eta\cdot f)(g) =\int_{G_{n,v}}f(gh^{-1})\eta(h)dh,\]
where $dh$ is a Haar measure on $G_{n,v}$ normalized so that the volume of $K_{n,v}$ is 1.

\begin{dfn}
	We say that a continuous right $K_{n,v}$-invariant function $f$ on $G_{n,v}$ (or on $G_{n,\adele}$) is a $\mathfrak{p}$-Hecke eigenfunction
	if $f$ is an eigenfunction under the action of $\mathcal{H}_{n,\mathfrak{p}}$.
\end{dfn}

\begin{dfn}
	We say that a Hermitian automorphic form $f\in\mathcal{A}_n(\rho,\frakn)$ is a Hecke eigenform
	if $f$ is a $\mathfrak{p}$-Hecke eigenfunction for any $\mathfrak{p}$ not dividing $\frakn$.
\end{dfn}

\section{Review of the differential operators on Hermitian modular forms}

Let $n_1,\ldots,n_d$ be positive integers such that $n_1\geq\cdots\geq n_d\geq 1$
and put $n = n_1+\cdots+n_d \geq 2$. We embed $\hus{n_1}^\bfa \times\cdots\times \hus{n_d}^\bfa$ in $\hus{n}^\bfa$ and
$G_{n_1}\times\cdots\times G_{n_d}$ in $G_{n}$ diagonally.

Let $(\rho_s,V_s)$ be a representation of $K_{n_s,\infty}^\bbc$ for $s=1,\ldots,d$,
and $\kappa=(\kappa_v)_{\va}$ a family of positive integers.

We will consider $V := V_{1}\otimes\cdots\otimes V_{d}$-valued differential operators $\mathbb{D}$
on scalar-valued functions on $\hus{n}^\bfa$, satisfying Condition (A) below:

\begin{cond}
	For any modular forms $F\in M_\kappa(\Gamma_K^{(n)})$, we have
	\[\mathrm{Res}(\mathbb{D}(F))\in \bigotimes_{i=1}^d M_{{\det}^\kappa\rho_{n_i}}(\Gamma_K^{(n_i)}),\]
	where $\mathrm{Res}$ means the restriction of a function on $\hus{n}^\bfa$ to
	$\hus{n_1}^\bfa \times\cdots\times \hus{n_d}^\bfa$.
\end{cond}

This Condition (A) corresponds to Case (I) in \cite{ibukiyama1999differential},
and the differential operators constructed for several vector-valued cases in \cite{Browning2024Constructing}.
The interpretation in the representation theory of $\mathbb{D}$ is done in \cite{Takeda2025pullback}.
Case (II) in \cite{ibukiyama1999differential} for Hermitian modular forms is discussed by Dunn \cite{Dunn2024Rankin}.

One could also formulate these operators, as in Ibukiyama's approach \cite{ibukiyama1999differential}, as holomorphic differential operators on the Hermitian upper half-space which are compatible with the actions of products of unitary groups of smaller rank.
Such a formulation would be independent of modular forms and would involve only complex analysis on the Hermitian upper half-space.
In this paper, however, we only apply these operators to Hermitian modular forms and automorphic forms.
For this reason, we state the condition in the above form, which directly records the transformation property after applying the operator and restricting to $\hus{n_1}^\bfa \times\cdots\times \hus{n_d}^\bfa$.

In the following, we recall the results in \cite{Takeda2025pullback}.

We put $\partial_Z=\left(\frac{\partial}{\partial Z_{ij}}\right)$.
Let $P_v(X)$ be a vector-valued polynomial on a space $M_n$ of degree $n$ variable matrices.
We will give the equivalent condition that the differential operator
$\mathbb{D}=P(\partial_{Z})=\prod_{\va}P_v(\partial_{Z_v})$ satisfies the Condition (A).
Let $L_{n,\kappa}=\bbc[X,Y]$ be the space of polynomials
in the entries of $(n,\kappa)$-matrices $X=(x_{ij})$ and $Y=(y_{ij})$ over $\mathbb{C}$.

\begin{dfn}
	We put
	\[\Delta_{ij}=\sum_{\nu=1}^{\kappa}\frac{\partial^2 }{\partial X_{i,\nu} \partial Y_{j,\nu}}\]
	and call this ``mixed Laplacian'' as in \cite{Ibukiyama2014Higher}.
	If a polynomial $f(X, Y) \in L_{n,\kappa}$ satisfies
	\[\Delta_{ij}f:=\sum_{s=1}^\kappa \frac{\partial^2f}{\partial X_{i,s}\partial Y_{j,s}}=0 \text{  for any  } i,j\in\{1,\ldots,n\}, \]
	we say that $f(X,Y)$ is pluriharmonic polynomial for $\rmu(\kappa)$.\\ We denote by
	$\calp_{n,\kappa}$ the set of all pluriharmonic polynomials for $\rmu(\kappa)$ in $L_{n,\kappa}$.
\end{dfn}

\begin{prop}[Corollary~16 of \cite{Takeda2025pullback}]\label{prop:diff}
	Let $n_1,\ldots,n_d$ be positive integers such that $n_1\geq\cdots\geq n_d\geq 1$
	and put $n = n_1+\cdots+n_d$.
	We take a family $(\bfk_s, \bfl_s)=(\bfk_{s,v},\bfl_{s,v})_{\va}$ of pairs of dominant integral weights such that $\ell(\bfk_{s,v})\leq n_s$, $\ell(\bfl_{s,v})\leq n_s$
	and $\ell(\bfk_{s,v})+\ell(\bfl_{s,v})\leq\kappa_v$ for all $s=1,\ldots,d$ and $\va$.

	Let $P_v(T)$ be a $\left(V_{n_1,\bfk_{1,v},\bfl_{1,v}}\otimes\cdots\otimes V_{n_d,\bfk_{d,v},\bfl_{d,v}}\right)$-valued polynomial on a space of degree $n$ variable matrices $M_n$ for $\va$.
	The differential operator $\mathbb{D}=P(\partial_{Z})=\prod_{\va}P_v(\partial_{Z_v})$ satisfies the Condition (A) for ${\det}^\kappa$
	and ${\det}^\kappa\rho_{n_1,\bfk_{1},\bfl_1}\otimes\cdots\otimes{\det}^\kappa\rho_{n_d,\bfk_d,\bfl_d}$ if and only if
	each $P_v(T)$ satisfies the following conditions:
	\begin{enumerate}
		\item If we put
		      \[\widetilde{P_v}(X_1,\ldots,X_d,Y_1,\ldots,Y_d)=
			      P_v\left(
			      \begin{pmatrix}
					      X_1\,^tY_1 & \cdots & X_1\,^tY_d \\
					      \vdots     & \ddots & \vdots     \\
					      X_d\,^tY_1 & \cdots & X_d\,^tY_d
				      \end{pmatrix}\right)\]
		      with $X_i, Y_i \in M_{n_i,\kappa_v}$,
		      then $\widetilde{P_v}$ is pluriharmonic for each $(X_i, Y_i)$.
		\item For $(A_i,B_i) \in K_{n_i,v}^\bbc:=\mathrm{GL}_{n_i}(\bbc)\times\mathrm{GL}_{n_i}(\bbc)$,
		      we have
		      \begin{align*}
			        & P_v\left(
			      \begin{pmatrix} A_1 & & \\&\ddots&\\ & &  A_d\end{pmatrix}
			      T
			      \begin{pmatrix} ^tB_1 & & \\&\ddots&\\ & &  ^tB_d\end{pmatrix}
			      \right)                                                                                                                 \\
			      = & \left(\rho_{n_1,\bfk_{1,v},\bfl_{1,v}}(A_1,B_1)\otimes\cdots\otimes\rho_{n_d,\bfk_{d,v},\bfl_{d,v}}(A_d,B_d)\right)P_v(T).
		      \end{align*}
	\end{enumerate}
\end{prop}
\begin{rem}
	For a Hermitian modular form regarded as a function on $\rmu(n,n)$,
	we can see that $P(\pi^+)=\prod_{\va}P(\pi^+_v)$ for $P(T)$ in the above proposition satisfies the same condition as Condition (A).
	Here $\pi^+_v=(\pi^+_{v,ij})$ is a matrix of the element in $\mathfrak{p}^+_{n,v}$.
	The proof of Proposition~\ref{prop:diff} is based on this interpretation.
	After regarding Hermitian modular forms as automorphic forms on $\rmu(n,n)$, a modular form of scalar weight gives a homomorphism from a lowest weight $(\mathfrak{g}_{n,\bbc},K_{n,\infty})$-module to the space of automorphic forms.
	Applying $P(\pi^+)$ and then restricting to $\rmu(n_1,n_1)\times\cdots\times\rmu(n_d,n_d)$ corresponds to a homomorphism between the restriction of this lowest weight module and the tensor product of the lowest weight modules attached to the target vector-valued weights.
	Howe duality for the Weil representation identifies this Hom-space with the space of $\rmu(\kappa)$-invariant pluriharmonic polynomials satisfying the required $K_{n_1,\infty}^\bbc\times\cdots\times K_{n_d,\infty}^\bbc$-equivariance.
	This representation-theoretic viewpoint explains why the two conditions in Proposition~\ref{prop:diff} are natural.
	Alternatively, one may prove the same criterion directly by complex analytic computations on the Hermitian upper half-space, following Ibukiyama's approach for Siegel modular forms \cite{ibukiyama1999differential}.
	(See \cite{Takeda2025pullback} for details.)
\end{rem}

Considering each place, we construct a polynomial $P(T)$ satisfying conditions (1) and (2) of the above proposition in concrete form below.
Therefore, the following discussion will not write $v$ for the places.

\section{Higher spherical pluriharmonic functions}

Let $n, \kappa$ be non-negative integers.
For a polynomial $P(T)\in \bbc[T]$ in the entry of $(n,n)$-matrix $T=(t_{ij})$ over $\mathbb{C}$,
we define the polynomial $\widetilde{P}(X,Y)\in L_{n,\kappa}$ in the entries of $(n,\kappa)$-matrices $X=(x_{ij})$ and $Y=(y_{ij})$ over $\mathbb{C}$ as
\begin{equation}\label{eq:tilde}
	\widetilde{P}(X,Y)=P(X^tY).
\end{equation}

We rephrase the condition (1) of the Proposition~\ref{prop:diff} for $\widetilde{P}(X,Y)$ as the condition for $P(T)$.
Since $t_{ij}=\sum_{s=1}^\kappa X_{i,s}Y_{j,s}$, we have
\[\frac{\partial \widetilde{P}}{\partial x_{i,\nu}}
	=\sum_{l=1}^n y_{l,\nu}\widetilde{\frac{\partial P}{\partial t_{i,l}}}.\]
Differentiating once more, we obtain
\[\frac{\partial^2 \widetilde{P}}{\partial x_{i,\nu} \partial y_{j,\nu}}
	=\left(\frac{\partial P}{\partial t_{ij}}\right)^\sim + \sum_{k,l=1}^n x_{k,\nu}y_{l,\nu}\left(\frac{\partial^2 P}{\partial t_{i,l}\partial t_{k,j}}\right)^\sim,\]
and summing this over $\nu$ gives
\[\Delta_{ij}\widetilde{P}:=\sum_{\nu=1}^{\kappa}\frac{\partial^2 \widetilde{P}}{\partial x_{i,\nu} \partial y_{j,\nu}}
	=(D_{ij}P)^\sim.\]
where we put
\begin{equation}\label{eq:Dij}
	D_{ij}=D^{(\kappa)}_{ij}=\kappa\partial_{ij}+\sum_{k,l=1}^n t_{kl}\partial_{i,l}\partial_{k,j}
\end{equation}
and call this also ``mixed Laplacian''.

When thinking in terms of $\Delta_{ij}$, $\kappa$ had to be a non-negative integer,
but by thinking in $D^{(\kappa)}_{ij}$, $\kappa$ can be considered a complex variable.

Let $\bfn=(n_1,n_2,\ldots,n_d)$ be a partition of $n$ such that $n_1\geq n_2\geq \cdots\geq n_d\geq 1$.
We put
\[I(\bfn)_s=\left\{(i,j)\in\bbz^2\mid n_1+n_2+\cdots+n_{s-1}+1\leq i,j \leq n_1+\cdots+n_s\right\}\]
for $s=1,\ldots,d$ and
\[I(\bfn)=\bigcup_{s=1}^dI(\bfn)_s.\]
For any $\kappa\in\bbc$, we put
\begin{align*}
	\calp^{(n)}(\kappa) & =\left\{ P(T) \in\bbc[T] \mid D^{(\kappa)}_{ii}P=0 \text{ for all } i \right\},                \\
	\calp^\bfn(\kappa)  & =\left\{ P(T) \in\bbc[T] \mid D^{(\kappa)}_{ij}P=0 \text{ for all } (i,j)\in I(\bfn) \right\}.
\end{align*}

We call the elements of $\calp^\bfn(\kappa)$ ``higher spherical pluriharmonic polynomials for the partition $\bfn$''.
Let $\mathcal{D}=(D_{ij})_{1\leq i,j\leq n}$ denote the $n\times n$ matrix of differential operators.
For $A=(a_{ij}), B=(b_{ij})\in \mathrm{GL}_n(\bbc)$, we denote by $(A\mathcal{D}^tB)_{ij}$ the $(i,j)$-th entry of the matrix $A\mathcal{D}^tB$,
that is, $(A\mathcal{D}^tB)_{ij}=\sum_{s,t=1}^n a_{i,s}D_{st}b_{jt}$.

Let $(\rho,V)$ be an algebraic representation of $\gn\times\gn$,
where we put $\gn=\mathrm{GL}_{n_1}(\bbc)\times\cdots\times\mathrm{GL}_{n_d}(\bbc)\subset \mathrm{GL}_n(\bbc)$.
\begin{dfn}
	We denote by $\calp^\bfn_\rho(\kappa)$ the space of $V$-valued polynomials $P_V(T)$ which
	satisfy the following two conditions.
	\begin{enumerate}
		\item Any components of $P_V(T)$ are in $\calp^\bfn(\kappa)$.
		\item For any $g_1,g_2\in \gn$, we have $P_V(g_1T^tg_2)=\rho(g_1,g_2)P_V(T)$.
	\end{enumerate}
\end{dfn}
Note that the conditions in the definition of $\calp^\bfn_\rho(\kappa)$
correspond to those of Proposition~\ref{prop:diff} when $\kappa$ is non-negative integer.

\begin{dfnprop}
	$\gn\times\gn$ acts on $\calp^\bfn(\kappa)$ by
	\[(g_1,g_2)P(T)=P(^tg_1Tg_2)\]
	for $g_1,g_2\in \gn$ and $P(T)\in \calp^\bfn(\kappa)$.
	This action is well-defined and $\calp^\bfn(\kappa)$ becomes a representation of $\gn \times \gn$.
\end{dfnprop}
\begin{proof}
	It can be computed in the same way as \cite[Proposition~2.4]{Ibukiyama2020generic}.
	We only note here that we have
	\[D_{ij}(P(^tATB))=((AD^tB)_{ij}P)(^tATB)\]
	for any $P(T)\in \bbc[T]$ and $A,B\in \mathrm{GL}_n(\bbc)$.
\end{proof}

\begin{rem}
	Let $(\rho',V')$ be a representation of $\gn\times\gn$ defined by
	\[\rho'(g_1,g_2)=\rho({}^tg_1,{}^tg_2)^{-1}.\]
	Then, the space $\calp^\bfn_\rho(\kappa)$ can be identified with
	$\left(\calp^\bfn(\kappa)\otimes V'\right)^{\gn\times\gn}$.
\end{rem}

\subsection{Inner product and Operators on $\bbc[T]$}
First, we introduce a Hermitian inner product on $\bbc[T]$ by
\[\left(P,Q\right)_\kappa=c_n(\kappa)\int_{{\herm{n}}_{>0}}\exp\left(-\frac{\tr(T)}{2}\right)P(T)\overline{Q(T)}\det(T)^{\kappa-n}dT,\]
where $c_n(\kappa)=2^{-\frac{n\kappa}{2}}\pi^{-\frac{n(n-1)}{2}}\left(\prod_{i=0}^{n-1}\Gamma(\kappa-i)\right)^{-1}$
and $dT$ is the Lebesgue measure on $\herm{n}\cong \bbr^{n^2}$.
If $\kappa$ is a positive integer $\geq n$, this integral is equal to
\[(2\pi)^{-n\kappa}\int_{M_{n,\kappa}(\bbc)}\exp\left(-\frac{\tr(XX^*)}{2}\right)P(XX^*)\overline{Q(XX^*)}dX,\]
where $dX$ is the Lebesgue measure on $M_{n,\kappa}(\bbc)$.
We can see that this inner product is absolutely convergent when $\Re(\kappa)>n-1$, and $(1,1)_\kappa=1$.
For proofs of these properties, see, for example, \cite[Lemma~2]{Takeda2025pullback} and \cite[\S3]{Ibukiyama2014Higher}.

We put $T^\nu=\prod_{i,j} t_{ij}^{\nu_{ij}}$ for $\nu=(\nu_{ij})\in\bbz^{n\times n}$.
If we set $\bfe_\kappa(P)=(P,1)_\kappa$,
then we can write $\left(P,Q\right)_\kappa=\bfe_\kappa(P\cdot \theta Q)$,
where $\theta Q=\sum_\nu \overline{C_\nu} (^t T)^\nu$ for $Q(T)=\sum_\nu C_\nu T^\nu$.

Let us return to the very first situation and consider the case when $\kappa$ is an integer $\geq n$.
We define the operator on $L_{n,\kappa}$ as follows:
\begin{align*}
	E_{ij}  & =\sum_{s=1}^\kappa x_{i,s}\frac{\partial }{\partial x_{j,s}}+\frac{\kappa}{2}\delta_{ij}, \\
	E'_{ij} & =\sum_{s=1}^\kappa y_{i,s}\frac{\partial }{\partial y_{j,s}}+\frac{\kappa}{2}\delta_{ij}, \\
	F_{ij}  & =\sum_{s=1}^\kappa x_{i,s}y_{j,s}.
\end{align*}
\begin{rem}
	Under the complexified Lie algebra, these operators generate a Lie algebra isomorphic to $\mathfrak{gl}_{2n}(\bbc)$ via the Weil representation (see \cite[Definition~8]{Takeda2025pullback}).
	In particular, these operators $\left<D_{ij},E_{ij},E'_{ij},F_{ij} \mid 1\leq i,j\leq n\right>_\bbc$ form a Lie algebra isomorphic to $\mathrm{Lie}(\rmu(n,n))$.
\end{rem}

By the identification \eqref{eq:tilde} defined earlier, the corresponding operators on $\bbc[T]$ are given as:
\begin{align*}
	E_{ij}  & \mapsto\sum_{s=1}^n t_{i,s}\frac{\partial }{\partial t_{j,s}}+\frac{\kappa}{2}\delta_{ij}, \\
	E'_{ij} & \mapsto\sum_{s=1}^n t_{s,i}\frac{\partial }{\partial t_{s,j}}+\frac{\kappa}{2}\delta_{ij}, \\
	F_{ij}  & \mapsto t_{ij},
\end{align*}
where $\delta_{ij}$ is the Kronecker delta function.
These corresponding operators on $\bbc[T]$ will also be denoted by the same notations.

\begin{prop}\label{prop:Ekappa}
	For any $\kappa\in\bbc$, $1\leq i,j \leq n$ and $P(T)\in \bbc[T]$, we have
	\[\bfe_\kappa(F_{ij}P)=\bfe_\kappa((E_{ij}+\frac{\kappa}{2}\delta_{ij})P)
		=\bfe_\kappa((E'_{ji}+\frac{\kappa}{2}\delta_{ij})P)=\bfe_\kappa((D_{ji}+\kappa\delta_{ij})P).\]
\end{prop}

\begin{proof}
	It can be proved in the same way as \cite[\S4 Proposition~1]{Ibukiyama2014Higher}, but we give an outline of the proof for the reader.
	We may assume that $\kappa\in\bbz_{\geq n}$.
	We put
	\[\bfe_{n,\kappa}(F)=(2\pi)^{-n\kappa}\int_{M_{n,\kappa}(\bbc)}\exp\left(-\frac{\tr(XX^*)}{2}\right)F(X,\overline{X})dX\]
	for $F(X,Y)\in L_{n,\kappa}$.
	Then we have $\bfe_{n,\kappa}(\widetilde{P})=\bfe_{\kappa}(P)$.
	Note that
	\[
		\bfe_{n,\kappa}(\frac{\partial F}{\partial x_{ij}})=\bfe_{n,\kappa}(y_{ij}F), \qquad
		\bfe_{n,\kappa}(\frac{\partial F}{\partial y_{ij}})=\bfe_{n,\kappa}(x_{ij}F),
	\]
	the same calculation as \cite[\S4 Proposition~1]{Ibukiyama2014Higher} yields
	\[\bfe_{n,\kappa}(F_{ij}F)=\bfe_{n,\kappa}((E_{ij}+\frac{\kappa}{2}\delta_{ij})F)
		=\bfe_{n,\kappa}((E'_{ji}+\frac{\kappa}{2}\delta_{ij})F)=\bfe_{n,\kappa}((\Delta_{ji}+\kappa\delta_{ij})F).\]
	Applying this to $F = \widetilde{P}$ with $P\in C[T]$, we obtain the desired formula for all $\kappa \in \bbz_{\geq n }$ and hence for all $\kappa$.
\end{proof}

\begin{prop}\label{prop:adj}
	The adjoint operators of $D_{ij},E_{ij},F_{ji}$ with respect to the action on the
	space $\bbc[T]$ with the inner product $(P,Q)_\kappa$ are given by
	\begin{align*}
		D_{ij}^*         & =D_{ji}-E_{ij}-E'_{ji}+F_{ij}, &
		E_{ij}^*         & =-E'_{ji}+F_{ji},              &
		E^{\prime*}_{ij} & =-E_{ji}+F_{ij},               &
		F_{ij}^*         & =F_{ji}.
	\end{align*}
\end{prop}

\begin{proof}
	Fourth equation is trivial and the first three equations can be proved by the same calculation as \cite[\S4 Proposition~2]{Ibukiyama2014Higher}.
	We may assume that $\kappa\in\bbz_{\geq n}$ again.
	If we define $\bfe_{n,\kappa}$ as above, for any $P,Q\in \bbc[T]$ we have
	\[(P,Q)_\kappa=\bfe_{n,\kappa}(\widetilde{P}\cdot \theta \widetilde{Q}),\]
	where $\theta$ is the operator on $L_{n,\kappa}$ defined by
	\[\theta F(X,Y)=\sum_{\nu,\mu} \overline{C_{\nu,\mu}} Y^\nu X^\mu\]
	for $F(X,Y)=\sum_{\nu,\mu} C_{\nu,\mu} X^\nu Y^\mu$.
	By direct calculation using Proposition~\ref{prop:Ekappa}, we have
	\begin{align*}
		\bfe_{n,\kappa}(F (\Delta_{ij}G)-(\Delta_{ij}F)G) & =\bfe_{n,\kappa}(F(-F_{ji}+E_{ji}+E'_{ij})G), \\
		\bfe_{n,\kappa}(F (E_{ij}G)+(E_{ij}F)G)           & =\bfe_{n,\kappa}(F(F_{ij})G),                 \\
		\bfe_{n,\kappa}(F (E'_{ij}G)+(E'_{ij}F)G)         & =\bfe_{n,\kappa}(F(F_{ji})G)
	\end{align*}
	for any $F,G\in L_{n,\kappa}$.
	Thus, we have
	\begin{align*}
		\bfe_{n,\kappa}((\Delta_{ij}F)\cdot\theta G) & =\bfe_{n,\kappa}(F \cdot \theta((\Delta_{ji}-E_{ij}-E'_{ji}+F_{ij})G)), \\
		\bfe_{n,\kappa}((E_{ij}F)\cdot\theta G)      & =\bfe_{n,\kappa}(F\theta((-E'_{ij}+F_{ji})G)),                          \\
		\bfe_{n,\kappa}((E'_{ij}F)\cdot\theta G)     & =\bfe_{n,\kappa}(F\theta((-E_{ij}+F_{ij})G)).
	\end{align*}
	Applying this to $F = \widetilde{P}$ and $G = \widetilde{Q}$ with $P,Q\in C[T]$, we obtain the desired equation for all $\kappa \in \bbz_{\geq n }$ and hence for all $\kappa$.
\end{proof}

We define the subspace $\bbc[T]_{\bfa, \bfb}$ of $\bbc[T]$ as
\[\bbc[T]_{\bfa,\bfb}=\left\{P\in\bbc[T]\mid P(\lambda_1T\lambda_2)
	=\lambda_1^\bfa\lambda_2^\bfb P(T) \text{ for all } \lambda_1, \lambda_2\in \text{diag}(\bbc^n)\right\}\]
for $\bfa,\bfb\in \bbz^n_{\geq0}$.
We put
\[\calp_{\bfa,\bfb}(\kappa)=\calp^{(n)}(\kappa)\cap\bbc[T]_{\bfa,\bfb}, \qquad
	\calp^\bfn_{\bfa,\bfb}(\kappa)=\calp^\bfn(\kappa)\cap\bbc[T]_{\bfa,\bfb}.\]
We introduce the following notation:
\begin{align*}
	\caln=\caln_n           & =M_n(\bbz_{\geq0}),                                                                                   \\
	\caln_0=\caln_{n,0}     & =\left\{\nu=(\nu_{ij})\in\caln \mid \nu_{ii}=0 \text{ for all } i  \right\},                          \\
	\caln(\bfa,\bfb)        & =\left\{\nu\in\caln\mid \nu\cdot\mathbf{1}=\bfa, \ ^t\mathbf{1}\cdot \nu=\bfb\right\},                \\
	\caln^\bfn_0(\bfa,\bfb) & =\left\{\nu=(\nu_{ij})\in\caln(\bfa,\bfb) \mid \nu_{ij}=0 \text{ for all } (i,j)\in I(\bfn) \right\}.
\end{align*}
Here, $\mathbf{1}$ is the column vector in $\bbz^n_{\geq0}$ such that all the components are 1.
When $\bfn=(1,1,\ldots,1)$, we often omit the subscript $\bfn$ of the notations.

Then, we have the direct sum decompositions of $\calp^{(n)}(\kappa)$ and $\calp^\bfn(\kappa)$ as follows:
\[
	\calp^{(n)}(\kappa) =\bigoplus_{\bfa,\bfb\in\bbz^n_{\geq0}}\calp_{\bfa,\bfb}(\kappa), \qquad
	\calp^\bfn(\kappa)   =\bigoplus_{\bfa,\bfb\in\bbz^n_{\geq0}}\calp^\bfn_{\bfa,\bfb}(\kappa).
\]
Moreover, this decomposition is orthogonal with respect to the inner product $(P,Q)_\kappa$.

\begin{cor}\label{cor:orthogonal}
	The spaces $\calp_{\bfa,\bfb}(\kappa)$ and $\calp_{\bfa',\bfb'}(\kappa)$ with distinct pairs of multi-indices $(\bfa,\bfb)$ and $(\bfa',\bfb')$ are orthogonal.
\end{cor}
\begin{proof}
	Let $\bfa=(a_1,\ldots,a_n)$, $\bfb=(b_1,\ldots,b_n)$, $\bfa'=(a_1',\ldots,a_n')$ and $\bfb'=(b_1',\ldots,b_n')$ be two multi-indices.

	Since we have
	\[D_{ii}-D_{ii}^*=E_{ii}-E_{ii}^*\]
	by Proposition~\ref{prop:adj},
	we have
	\[0=(D_{ii}P,Q)_\kappa-(P,D_{ii}Q)_\kappa
		=(E_{ii}P,Q)_\kappa-(P,E_{ii}Q)_\kappa
		=(a_i-a_i')(P,Q)_\kappa\]
	for any $P\in\calp_{\bfa,\bfb}(\kappa)$ and $Q\in\calp_{\bfa',\bfb'}(\kappa)$ and $i\in\left\{1,\ldots,n\right\}$.
	Similarly, since we also have $D_{ii}-D_{ii}^*=E'_{ii}-E^{\prime*}_{ii}$, we obtain $(b_j-b_j')(P,Q)_\kappa=0$ for any $j\in\left\{1,\ldots,n\right\}$.
	Therefore, if $(\bfa,\bfb)\neq(\bfa',\bfb')$, we have $(P,Q)_\kappa=0$.
\end{proof}

\subsection{Two bases of $\calp^{(n)}(\kappa)$}

In the following, we will construct two ``good'' bases of $\calp^{(n)}(\kappa)$ and $\calp^{\bfn}(\kappa)$ that are compatible with $D_{ij}$ according to the method of \cite{Ibukiyama2014Higher} and \cite{Ibukiyama2020generic}.

We note that there is an isomorphism $\bbc[T]_{\bfa,\bfb}\cong\bbc^{\caln(\bfa,\bfb)}$
given by mapping a polynomial to its set of coefficients.
We consider the map $\Phi: \calp^\bfn_{\bfa,\bfb}(\kappa)\rightarrow\bbc^{\caln_0(\bfa,\bfb)}$,
which is the composition of the inclusion $\calp^\bfn_{\bfa,\bfb}(\kappa)\subset\bbc[T]_{\bfa,\bfb}\cong\bbc^{\caln(\bfa,\bfb)}$
and the projection $\bbc^{\caln(\bfa,\bfb)}\rightarrow\bbc^{\caln_0(\bfa,\bfb)}$.
We put $\Xi_{\bfa,\bfb}=\Xi=\bbz\cap\bigcup_{i=1}^d[2-(a_i+b_i),\ 1-\max\left\{a_i,b_i\right\}]$.
\begin{prop}\label{prop:isom}
	If $\kappa\in \bbc\backslash\Xi_{\bfa,\bfb}$,
	$\Phi: \calp^\bfn_{\bfa,\bfb}(\kappa)\rightarrow\bbc^{\caln_0(\bfa,\bfb)}$ is an isomorphism.
\end{prop}
\begin{proof}
	The proof can be done as \cite[Theorem~1]{Ibukiyama2014Higher}.
	Let $P(T)=\sum_{\nu\in\bbc^{\caln(\bfa,\bfb)}} c_\nu T^\nu$ be an element of $\calp^\bfn_{\bfa,\bfb}(\kappa)$.
	Since we can easily see that
	\begin{equation}\label{recursion}
		D_{ii}(T^\nu)=\nu_{ii}(a_i+b_i-\nu_{ii}+\kappa-1)T^{\nu-e_{ii}}
		+\sum_{k,l\neq i}\nu_{i,l}\nu_{k,i}T^{\nu-e_{i,l}-e_{k,i}+e_{kl}},
	\end{equation}
	the differential equations $D_{ii}P = 0$ yield recursion formulas for the coefficients of $P(T)$.
	when $\kappa\not\in \Xi$,
	we have $a_i+b_i-\nu_{ii}+\kappa-1\neq0$, and $P(T)$ is determined only by the coefficients of degree $\nu \in\caln_0(\bfa,\bfb)$.

	Conversely, the surjectivity of $\Phi$ follows from the following stronger claim:
	\begin{claim}
		If we put $K_i(\bfa,\bfb)=\ker(D_{1,1})\cap\ker(D_{2,2})\cap\cdots\cap\ker(D_{ii})\subset\bbc[T]_{\bfa,\bfb}$, we have
		\[\dim K_i = N_{i}(\bfa,\bfb)=\left|\left\{\nu\in\caln(\bfa,\bfb)\mid\nu_{1,1}=\cdots=\nu_{ii}=0\right\}\right|\]
		for $i\in\left\{0,\ldots,n\right\}$.
	\end{claim}
	$\dim K_{i}(\bfa,\bfb)\leq N_{i}(\bfa,\bfb)$ follows from the fact the above recursion formulas.
	Since differential operators $D_{ij}$ are mutually commutative, $D_{ii}$ sends $K_i(\bfa,\bfb)$ to $K_i(\bfa-e_i,\bfb-e_i)$,
	where $e_i$ denotes the vector with 1 in the $i$-th place and 0 elsewhere.
	If we therefore assume inductively that $\dim K_{i-1}(\bfa,\bfb) = N_{i-1}(\bfa,\bfb)$ for all $\bfa$ and $\bfb$,
	\begin{align*}
		\dim K_{i}(\bfa,\bfb) & =\dim (D_{ii}:K_{i-1}(\bfa,\bfb)\rightarrow K_{i-1}(\bfa-e_i,\bfb-e_i)) \\
		                      & \geq\dim K_{i-1}(\bfa,\bfb)-\dim K_{i-1}(\bfa-e_i,\bfb-e_i)             \\
		                      & = N_{i-1}(\bfa,\bfb) - N_{i-1}(\bfa-e_i,\bfb-e_i)                       \\
		                      & =N_{i}(\bfa,\bfb).
	\end{align*}
	Hence the claim holds, and surjectivity of $\Phi$ is established.
\end{proof}
\begin{cor}\label{cor:decomposition}
	For any $\kappa\in\bbc\backslash\bbz_{<0}$, we have
	\[\bbc[T]=\calp^{(n)}(\kappa)[t_{1,1},t_{2,2},\ldots,t_{n,n}].\]
	More precisely, we have a direct sum decomposition
	\[\bbc[T]_{\bfa,\bfb}=\bigoplus_{0\leq\mathbf{m}\leq\min\left\{\bfa,\bfb\right\}}\delta(T)^\mathbf{m}\calp_{\bfa-\mathbf{m},\bfb-\mathbf{m}}(\kappa),\]
	where the sum ranges over $\left\{\mathbf{m}=(m_i)\in\bbz^n_{\geq0} \mid 0\leq m_i\leq \min\left\{a_i,b_i\right\}\right\}$
	and $\delta(T)^\mathbf{m}=\prod_{i=1}^n t_{ii}^{m_i}$.
\end{cor}
\begin{proof}
	See Corollary to Theorem~1 in \cite{Ibukiyama2014Higher}.
\end{proof}

We begin by constructing bases of \(\mathcal{P}_{\bfa,\bfb}(\kappa)\).
Let $\Pi^{(\kappa)}_{\bfa,\bfb}:\bbc[T]_{\bfa,\bfb}\rightarrow\calp_{\bfa,\bfb}(\kappa)$ be a projection with respect to
the direct sum decomposition in the Corollary~\ref{cor:decomposition}.
That is, $\Pi^{(\kappa)}_{\bfa,\bfb}$ sends any polynomial $P(T)\in\bbc[T]_{\bfa,\bfb}$ to the unique polynomial in $\calp_{\bfa,\bfb}(\kappa)$ which is congruent to $P$ modulo the ideal $(t_{1,1}, \ldots, t_{n,n})$ of $\bbc[T]$.
$\Pi^{(\kappa)}_{\bfa,\bfb}$ can be written explicitly as follows.
\begin{lem}
	We have
	\[\Pi^{(\kappa)}_{\bfa,\bfb}=\prod_{i=1}^n \left(\sum_{0\leq j \leq \min\left\{a_i,b_i\right\}}\frac{t_{ii}^jD_{ii}^j}{j!(a_i+b_i+\kappa-2)_j}\right),\]
	where $(a)_i=a(a-1)\cdots(a-i+1)$.
\end{lem}
\begin{proof}
	See, for example, \cite[p.25]{Ibukiyama2014Higher}.
\end{proof}

As a Corollary to the theorem, we obtain the first basis of $\calp_{\bfa,\bfb}(\kappa)$.
\begin{dfnprop}\label{cor:monomial}
	We assume that $\kappa\in\bbc\backslash\Xi$.
	Let $P^M_{\nu,\kappa}(T)$ be a polynomial in $\calp_{\bfa,\bfb}(\kappa)$ defined by
	\[P^M_{\nu,\kappa}(T)=\Pi^{(\kappa)}_{\bfa,\bfb}(T^\nu)\]
	for $\caln_0(\bfa,\bfb)$.
	Then, these polynomials $\left\{P^M_{\nu,\kappa} \right\}_{\nu\in \caln_0(\bfa,\bfb) }$ form a basis of $\calp_{\bfa,\bfb}(\kappa)$.
	The basis $\left\{P^M_{\nu,\kappa} \right\}$ is called the ``monomial basis'' of $\calp_{\bfa,\bfb}(\kappa)$.
\end{dfnprop}
\begin{rem}
We can also define the monomial basis of $\calp_{\bfa,\bfb}(\kappa)$ by the following characterization:
	For each multi-index $\nu\in\caln_0(\bfa,\bfb)$, there exists a unique polynomial $P^M_{\nu,\kappa}\in\calp_{\bfa,\bfb}(\kappa)$ such that
	\[P^M_{\nu,\kappa}(T)=T^\nu+Q(T),\]
	where $Q(T)|_{t_{1,1}=t_{2,2}=\cdots=t_{n,n}=0}=0$.
\end{rem}

The second basis is compatible with the operators $D_{ij}$ and can be constructed using the same method as in the proof of Proposition~\ref{prop:isom}.

\begin{dfnprop}\label{prop:descending}
	If $\kappa \in \bbc\backslash\bbz_{<n}$,
	there exists a unique basis $\left\{P^D_{\nu,\kappa}\mid\nu\in\caln_0(\bfa,\bfb)\right\}$ of $\calp_{\bfa,\bfb}(\kappa)$ characterized by the properties
	\[D_{ij}P_{\nu,\kappa}^D=P_{\nu-e_{ij}, \kappa}^D \ (i\neq j) \quad \text{and} \quad P_0^D=1.\]
	Here, when $\nu_{ij}=0$, we put $P_{\nu-e_{ij},\kappa}^D=0$.
	The basis $\left\{P^D_{\nu,\kappa}(T) \right\}$ is called the ``descending basis'' of $\calp_{\bfa,\bfb}(\kappa)$.
\end{dfnprop}
\begin{proof}
	See \cite[p.28-29, \S9]{Ibukiyama2014Higher}.
\end{proof}

\begin{prop}
	\begin{enumerate}
		\item For any two multi-indices $\mu$ and $\nu$ such that $\nu\cdot\mathbf{1}=\mu\cdot\mathbf{1}$ and $^t\mathbf{1}\cdot\nu={}^t\mathbf{1}\cdot\mu$,
		      we have
		      \[(P_{\mu,\kappa}^M,P_{\nu,\kappa}^M)_\kappa=D^\nu(P_{\mu,\kappa}^M),\]
		      where we put $D^\nu=\prod_{ij}D_{ij}^{\nu_{ij}}$.

		\item The monomial and descending bases of $\calp_{\bfa,\bfb}(\kappa)$ are dual to one another.
		      That is, for any two multi-indices $\mu$ and $\nu$, we have
		      \[(P_{\mu,\kappa}^M,P_{\nu,\kappa}^D)_\kappa=\delta_{\mu,\nu}.\]
	\end{enumerate}
\end{prop}
\begin{proof}
	See Corollary to \S9 Proposition~2 in \cite{Ibukiyama2014Higher} for (1)
	and Theorem~6 in \cite{Ibukiyama2014Higher} for (2).
\end{proof}

\begin{ex}[$n=2$]
	In this case, we have $\caln_0(\bfa,\bfb)=1$ if $a_1=b_2,\ a_2=b_1$ and $\caln_0(\bfa,\bfb)=0$ otherwise.
	Then consider the case $a_1=b_2\leq a_2=b_1$.
	By computing the recursion formula \eqref{recursion}, we can easily see that
	$P^M_{\nu,\kappa}$ is given by
	\[P^M_{\nu,\kappa}=\sum_{i=0}^{a_1}\frac{(-1)^i(a_1)_{(i)}\cdot(a_2)_{(i)}}{i!(a_1+a_2+\kappa-2)_{(i)}}T^{\nu_i},\]
	where $\nu_i=\begin{pmatrix}i&a_1-i\\a_2-i&i\end{pmatrix}\in\caln(\bfa,\bfb)$ for $0\leq i\leq a_1$.
\end{ex}

Then, we will construct the monomial and descending bases of $\calp^{\bfn}(\kappa)$ for general $\bfn$.
The descending basis of $\calp^{\bfn}(\kappa)$ consists of exactly the same polynomials $P^D_{\nu,\kappa}(T)$.

\begin{dfnprop}
	We assume that $\kappa\in\bbc\backslash\bbz_{<n}$.
	The polynomial $P^D_{\nu,\kappa}(T)\in \calp_{\bfa,\bfb}(\kappa)$ defined in the Definition-Proposition~\ref{prop:descending} is an element of $\calp^\bfn_{\bfa,\bfb}(\kappa)$ if and only if $\nu\in\caln^\bfn_0(\bfa,\bfb)$.
	In particular, $\left\{P^D_{\nu,\kappa}(T) \mid \nu\in\caln^\bfn_0(\bfa,\bfb) \right\}$ is a basis of $\calp^{\bfn}_{\bfa,\bfb}(\kappa)$, which is called the descending basis of $\calp_{\bfa,\bfb}^{\bfn}(\kappa)$.
\end{dfnprop}

The monomial basis of $\calp^{\bfn}(\kappa)$ must be slightly modified from the monomial basis of $\calp^{(n)}(\kappa)$
because $P_\nu^M(T)$ is generally not an element of $\calp^{\bfn}_{\bfa,\bfb}(\kappa)$ even if $\nu$ is in $\caln^\bfn_0(\bfa,\bfb)$.

We write the block decomposition of $T$ for the
partition $\bfn$ as $T = (T_{p,q})$, where $T_{p,q}$ is an $(n_p,n_q)$-matrix.

\begin{prop}
	Assume that $\kappa\in\bbc\backslash\bbz_{<n}$ and that $(*,*)_\kappa$ is positive definite.
	Then, for each multi-index $\nu\in\caln^\bfn_0(\bfa,\bfb)$,
	there exists a unique polynomial $P^{M,\bfn}_{\nu,\kappa}\in\calp^{\bfn}_{\bfa,\bfb}(\kappa)$ such that
	\[P^{M,\bfn}_{\nu,\kappa}(T)=T^\nu+Q(T),\]
	where $Q(T)|_{T_{1,1}=T_{2,2}=\cdots=T_{n,n}=0}=0$.
	These polynomials $\left\{P^{M,\bfn}_{\nu,\kappa} \right\}_{\nu\in \caln^\bfn_0(\bfa,\bfb) }$ form a basis of $\calp^{\bfn}_{\bfa,\bfb}(\kappa)$, which is called the monomial basis of $\calp_{\bfa,\bfb}^{\bfn}(\kappa)$.
\end{prop}
\begin{proof}
	See \cite[Theorem~4.3, Lemma~4.4]{Ibukiyama2020generic} for the proof.
\end{proof}
\begin{rem}
	$P^{M,\bfn}_{\nu,\kappa}$ can be written concretely as follows, in the same way as \cite[Remark~4.5]{Ibukiyama2020generic}.
	Let $\{c_\mu\}_{\mu\in\caln_0(\bfa,\bfb)\backslash\caln^\bfn_0(\bfa,\bfb)}$ be the solution of the following linear equations:
	\[\sum_{\mu\in\caln_0(\bfa,\bfb)\backslash\caln^\bfn_0(\bfa,\bfb)}c_\mu(P^{M}_{\xi,\kappa}(T),P^{M}_{\mu,\kappa}(T))_\kappa=-(P^{M}_{\xi,\kappa}(T),P^{M}_{\nu,\kappa}(T))_\kappa\]
	for all $\kappa\in\caln_0(\bfa,\bfb)\backslash\caln^\bfn_0(\bfa,\bfb)$.
	Then, we have
	\[P^{M,\bfn}_{\nu,\kappa}(T)
		=P^{M}_{\nu,\kappa}(T)+\sum_{\mu\in\caln_0(\bfa,\bfb)\backslash\caln^\bfn_0(\bfa,\bfb)}
		c_\mu P^{M}_{\mu,\kappa}(T).\]
\end{rem}

\subsection{A generating function for the descending basis}
We fix an positive integer $n$ and a partition $\bfn=(n_1,\ldots,n_d)$ of $n$.
Let $T,X$ be $(n,n)$-matrices of variables.
We write $X$ by matrix blocks as $X=(X_{p,q})$, where $X_{p,q}$ is a $(n_p,n_q)$-matrix for $1\leq p,q\leq d$.
We assume that $X_{p,p}=0$ for all $p\in\left\{1,\ldots,d\right\}$.

In this subsection, we will construct a generating function
\[G^\bfn(X,T)=\sum_{\nu\in\caln_0^\bfn}P_\nu(T)X^\nu\in\calk[T][[X]]\]
for the descending basis of $\calp^{\bfn}(\kappa)$, that is, which satisfies
\begin{equation}\label{eq:descending_generating}
	D_{ij}P_\nu(T)=c_\nu P_{\nu-e_{ij}}(T)
\end{equation}
with some constants $c_\nu$ for any multi-index $\nu\in\caln_0^\bfn$.
The calculation method is basically the same as in \cite{Ibukiyama2014Higher}.
However, in \cite{Ibukiyama2014Higher},
there is an assumption that X and T are symmetric matrices,
and for this reason the subscripts and coefficients are slightly different,
so we will explain in detail for the reader.

We put $\calk=\bbc(\kappa)$
We denote by $\calk[[X]]$ the vector space of formal power series in the components of $X$.
Let $\lambda$ be an independent variable and set
\[ P=P_\lambda=X^{-1}+\lambda ^tT ,\qquad \Delta=\Delta_\lambda=\det(P_\lambda).\]
We define polynomials $\sigma_i (X{}^tT)$ in $t_{ij}$ and $x_{ij}$ by the relation
\begin{equation}\label{eq:sigma}
	\Delta_\lambda=\det(X^{-1})\sum_{i=0}^n\sigma_i(X{}^tT)\lambda^i.
\end{equation}
Here, we have $\sigma_0=1$.
We put \[V=\calk[[s_1,s_2,\ldots]], \qquad V_k=V_{\deg=k}.\] and the degrees are determined by $\deg(s_a)=a$.
For $F\in V$, we define the function $\widetilde{F}\in\calk[T][[X]]$ as
\[\widetilde{F}(X,T)=F(\sigma_0(X{}^tT),\sigma_1(X{}^tT),\ldots,\sigma_n(X{}^tT),0, 0,\ldots).\]
We set $\partial_a={\partial}/{\partial s_a}$ and $\partial_{ij}={\partial}/{\partial t_{ij}}$.
Under the map $F\mapsto \widetilde{F}$, calculate what $D_{ij}$ in \eqref{eq:Dij} corresponds to.
For convenience, we define coefficients $\epsilon(m), \epsilon^\pm_{a,b}(m)\in\left\{0,\pm 1\right\}$
for $a,b,m\in\bbz$ as follows:
\begin{align*}
	\epsilon(m)         & =\begin{cases}1 & \text{if } m\geq0, \\
             0 & \text{if } m<0,\end{cases}     \\
	\epsilon^+_{a,b}(m) & =\begin{cases}+1 & \text{if } a,b\geq m, \\
             -1 & \text{if } a,b < m,   \\
             0  & \text{otherwise },\end{cases} \\
	\epsilon^-_{a,b}(m) & =\begin{cases}+1 & \text{if } b<m\leq a, \\
             -1 & \text{if } a<m\leq b, \\
             0  & \text{otherwise }.\end{cases} \\
\end{align*}

\begin{prop}
	For $p\geq 0$, we define a second order differential operator $\call_p$ on $V$ by
	\[\call_p = (\kappa+1-p)\partial_p+\sum_{a,b\geq 1}\epsilon^+_{a,b}(p)s_{a+b-p}\partial_a\partial_b,\]
	where $s_0=1$ and $s_p=0$ if $p<0$.
	Then, we have
	\[D_{ij}(\widetilde{F})=\sum_{p=1}^n\partial_{ij}(\sigma_p(X{}^tT))\widetilde{\call_p(F)}.\]
\end{prop}

\begin{proof}
	Since both sides of the equation are second order derivatives,
	it is sufficient to prove their equality for monomials of degree $\leq 2$ in $\calk[s_1, s_2, \ldots ] $, i.e., 1, $s_a$, $s_as_b$ ($a, b \geq 1$).
	The first case is trivial.
	We should show that
	\begin{equation}\label{eq:second}
		D_{ij}(\sigma_p)=(\kappa+1-p)\partial_{ij}(\sigma_p)
	\end{equation}
	for the second case,
	and
	\begin{equation}\label{eq:third}
		D_{ij}(\sigma_a\sigma_b)=D_{ij}(\sigma_a)\sigma_b+\sigma_aD_{ij}(\sigma_b)
		+2\sum_{p=1}^n \epsilon^+_{a,b}(p)\partial_{ij}(\sigma_p)\sigma_{a+b-p}
	\end{equation}
	for the third case.
	We can easily see that \[\partial_{ij}(\Delta_\lambda)=\lambda\Delta_\lambda^{j;i},\]
	where $\Delta_\lambda^{j;i}$ is the determinant of the $(j,i)$-cofactor of $P_\lambda$.
	Thus, we have
	\[\left(1-\lambda\frac{d}{d\lambda}\right)\partial_{ij}(\Delta_\lambda)
		=\left(1-\lambda\frac{d}{d\lambda}\right)\left(\lambda\Delta_\lambda^{j;i}\right)
		=-\lambda^2\sum_{k,l=1}^nt_{kl}\Delta_\lambda^{j,l;i,k}
		=\lambda^2\sum_{k,l=1}^nt_{kl}\Delta_\lambda^{j,l;k,i},\]
	where $\Delta_\lambda^{j,l;i,k}$ is the determinant of the $(n-2,n-2)$-matrix obtained from $P$ by omitting
	its $j$-th and $l$-th rows and $i$-th and $k$-th columns, multiplied by
	$(-1)^{j+l+i+k}\mathrm{sgn}(j - l)\mathrm{sgn}(i - k) $
	(and $\Delta_\lambda^{i,k;j,l} = 0$ if $i = k$ or $j = l$).
	Therefore we have,
	\[(D_{ij}-\kappa\partial_{ij})\Delta_\lambda
		=\sum_{k,l=1}^nt_{kl}\partial_{i,l}\partial_{k,j}(\Delta_\lambda)
		=\lambda^2\sum_{k,l=1}^nt_{kl}\Delta_\lambda^{j,l;k,i}
		=\left(1-\lambda\frac{d}{d\lambda}\right)\partial_{ij}(\Delta_\lambda),\]
	that is,
	\[D_{ij}\Delta_\lambda=\left(\kappa+1-\lambda\frac{d}{d\lambda}\right)\partial_{ij}(\Delta_\lambda)\]
	and comparing the coefficients of $\lambda^p$ in both sides, we obtain \eqref{eq:second}.

	Next, we show \eqref{eq:third}.
	Note that $\partial_{ij}(\Delta_\lambda)=\lambda\Delta_\lambda(P_\lambda^{-1})_{ij}$,
	so for independent variables $\lambda$ and $\mu$, we have
	\begin{align*}
		D_{ij}(\Delta_\lambda\Delta_\mu)-D_{ij}(\Delta_\lambda)\Delta_\mu- & \Delta_\lambda D_{ij}(\Delta_\mu)
		=\sum_{k,l=1}^nt_{kl}\left(\partial_{i,l}(\Delta_\lambda)\partial_{k,j}(\Delta_\mu)
		-\partial_{k,j}(\Delta_\lambda)\partial_{i,l}(\Delta_\mu)\right)                                                                                                                                   \\
		                                                                   & =\lambda\mu\Delta_\lambda\Delta_\mu\sum_{k,l=1}^n\left((P_\lambda^{-1})_{i,l}t_{kl}(P_\mu^{-1})_{k,j}
		+(P_\lambda^{-1})_{k,j}t_{kl}(P_\mu^{-1})_{i,l}\right)                                                                                                                                             \\
		                                                                   & =\lambda\mu\Delta_\lambda\Delta_\mu\left(P_\lambda^{-1}{}^tTP_\mu^{-1}+P_\mu^{-1}{}^tTP_\lambda^{-1}\right)_{ij}              \\
		                                                                   & =-\lambda\mu\Delta_\lambda\Delta_\mu\left(\frac{P_\lambda^{-1}-P_\mu^{-1}}{\lambda-\mu}\right)_{ij}                           \\
		                                                                   & =-\frac{1}{\lambda-\mu}\left(\mu\Delta_\mu\partial_{ij}(\Delta_\lambda)-\lambda\Delta_\lambda\partial_{ij}(\Delta_\mu)\right) \\
		                                                                   & =\det(X)^{-2}\sum_{p,q=0}^n\sigma_q\partial_{ij}(\sigma_p)
		\frac{\lambda^{q+1}\mu^p-\mu^{q+1}\lambda^p}{\lambda-\mu}.
	\end{align*}
	Here, fourth equality holds because ${}^tT=(P_\lambda-P_\mu)/(\lambda-\mu)$.
	Comparing the coefficients of $\lambda^a\mu^b$ in both sides, we obtain \eqref{eq:third}.
\end{proof}

The following corollary is an immediate consequence of the above proposition.
\begin{cor}\label{cor:DW}
	Let $W=\cap_{p\geq 2}\ker(\call_p)\in V$.
	Then, for any $F\in W$, we have $D_{ii}(\widetilde{F})=0$ and $D_{ij}(\widetilde{F})=x_{ij}\widetilde{\call_1(F)}$ for $i\neq j$.
\end{cor}

\begin{prop}
	The operators $\call_p$ satisfy the commutation formula
	\[\left[\call_{p},\call_{q}\right]
		=-\sum_{\substack{a,b\geq1\\a+b=p+q}}\epsilon^-_{p,q}(b)\partial_{a}\call_b \qquad (p,q\geq 1).\]
\end{prop}

\begin{proof}
	This can be done by a direct calculation as \cite[\S9 Proposition~2]{Ibukiyama2014Higher}.
\end{proof}

\begin{cor}\label{cor:subspace}
	The subspace $W$ of $V$ is mapped into itself by $\call_1$.
	More generally, the space $W\left<m\right>:=\cap_{p>m}\ker(\call_p)\subset V$  is mapped into itself by $\call_m$
	for all $m\geq 1$.
\end{cor}

\begin{prop}
	The space $W_k = W \cap V_k$ is one-dimensional for each $k \geq 0$. If $G_k$ is a non-zero
	element of $W_k$ for every $k \geq 0$,
	then $\call_1(G_k)$ is a non-zero multiple of $G_{k-1}$ for all  $k \geq 1$.
	If we set
	$G =\sum_{k \geq 0} G_k \in W$, then the polynomials $P_\nu\in \calk[T]$ defined as
	$\widetilde{G} =\sum_{\nu} P_\nu(T)X^\nu$ satisfy \eqref{eq:descending_generating}.
\end{prop}
\begin{proof}
	By the same calculation as \cite[Theorem~10]{Ibukiyama2014Higher}, we have the isomorphism
	\[\begin{array}{ccc}
			W\left<m\right> & \stackrel{\sim}{\longrightarrow} & \calk[[s_1,\ldots,s_m]]      \\
			F               & \longmapsto                      & F(s_1,\ldots,s_m,0,0,\ldots) \\
		\end{array}\]
	for any $m\geq 0$.
	In particular, we have $W \cong \calk[[s_1]]$ and  $W_k$ is one-dimensional for each $k \geq 0$.
	From the definition of $\call_1$ and the Corollary~\ref{cor:subspace}, we have $\call_1(W_k)\subset W_{k-1}$.
	And since $W\left<0\right>\cap V_k\cong\calk\cap V_k=\left\{0\right\}$, $\call_1(G_k)$ is a non-zero element of $W_{k-1}$ for all  $k \geq 1$.
	If we write $\call_1(G_k)=c_kG_{k-1}$ with some constants $c_k\in\calk$,
	we have $D_{ii}(\widetilde{G}_k)=0$ and $D_{ij}(\widetilde{G_k})=c_kx_{ij}\widetilde{G}_{k-1}$ for $i\neq j$ by the Corollary~\ref{cor:DW}.
	In particular, we have $D_{ii}(P_\nu)=0$ and $D_{ij}(P_\nu)=c_{|\nu|} P_{\nu-e_{ij}}$ for any $\nu\neq0$,
	where $|\nu|=\sum_{ij}\nu_{ij}$.
\end{proof}
Let $G =\sum_{k \geq 0} G_k $ be a element of $W$ as in the above proposition.
We put
\[G^{(n)}(s_1,\ldots,s_n)=G(s_1,\ldots,s_n,0,0,\ldots)\]
and expand it as a power series
\[G^{(n)}(s_1,\ldots,s_n)=\sum_{r=0}^\infty g_r^{(n)}\frac{s^r_n}{r!}\]
in $s_n$.
For $G^{(n)}$, $\call_n$ is equal to  $(\kappa+1-n)\partial_n+s_n\partial_n^2-\calm_n$,
where \[\calm_n=\sum_{\substack{0< a,b< n\\a+b\geq n}}s_{a+b-n}\partial_a\partial_b.\]
Since $\call_n(G^{(n)})=0$, comparing the coefficients of $s_n^r$ on both sides of the equation gives
\[(\kappa-n+r+1)g_{r+1}^{(n)}=\calm_n(g_r^{(n)})\]
for $r\geq 0$.
We note that $g_0^{(n)}=G^{(n-1)}$, we have
\[g_{r}^{(n)}=\frac{1}{(\kappa-n+1)^{(r)}}(\calm_n)^r(G^{(n-1)})\]
where $(x)^{(r)}=x(x+1)\cdots(x+r-1)$ is ascending Pochhammer symbol.
Since the operators $\calm_n$ and (multiplication by) $s_n$ commute,
\[G^{(n)}=\mathbb{J}_{\kappa-n}(s_n\calm_n)(G^{(n-1)})\]
where
\[\mathbb{J}_\nu(x)=\sum_{r=0}^\infty\frac{x^r}{r!(\nu+1)^{(r)}}.\]
Induction on $n$ gives
\begin{equation}\label{eq:generating}
	G^{(n)}=\mathbb{J}_{\kappa-n}(s_n\calm_n)\mathbb{J}_{\kappa-n+1}(s_{n-1}\calm_{n-1})\cdots\mathbb{J}_{\kappa-2}(s_2\calm_2)(G^{(1)}(s_1)).
\end{equation}
Here, we can take $G^{(1)}(s_1)$ as an arbitrary function.
Then $\widetilde{G^{(n)}}(X,T)$ is a generating function for the descending basis of $\calp^{(n)}(\kappa)$.
Summarizing, we have:
\begin{thm}\label{thm:generating}
	Let \(\widetilde{G^{(n)}}\) be the function given by \eqref{eq:generating}, associated with an arbitrary \(G^{(1)}(s_1) \in \mathcal{K}[[s_1]]\). Then \(\widetilde{G^{(n)}}\) is a generating function for the descending basis of \(\mathcal{P}^{\mathbf{n}}(\kappa)\).
\end{thm}

\subsection{Example of the generating function}

For $n=1$, $G^{(1)}$ can be taken as an arbitrary function as mentioned above.

For $n=2$, if we take $G^{(1)}=(1-s_1/2)^{3-2\kappa}$,
we have
\begin{align*}
	G^{(2)}(s_1,s_2) & =\mathbb{J}_{\kappa-2}(s_2\calm_2)(G^{(1)}(s_1))                                          \\
	                 & =\sum_{r=0}^\infty\frac{(s_2\calm_2)^r}{r!(\kappa-1)^{(r)}} (1-s_1/2)^{3-2\kappa}         \\
	                 & =\sum_{r=0}^\infty\binom{3/2-\kappa}{r} (-s_2)^r\left\{(1-s_1/2)^2\right\}^{3/2-\kappa-r} \\
	                 & =\frac{1}{\left\{(1-s_1/2)^2-s_2\right\}^{\kappa-3/2}}.
\end{align*}

For $n\geq3$, it is difficult to get the explicit form of the generating function by using \eqref{eq:generating}.
However, the expanded form can be obtained somewhat more easily using a recursion formula.
Let $G^{(3)}=\sum_{a,b,c}A_{a,b,c}\dfrac{s_1^a}{a!}\dfrac{s_2^b}{b!}\dfrac{s_3^c}{c!}$ be a function in $W\subset\calk[[s_1,s_2,s_3]]$.
We have
\begin{align*}
	\call_2 & =(\kappa-1)\partial_2+s_2\partial_2^2+2s_3\partial_2\partial_3-\partial_1^2, \\
	\call_3 & =(\kappa-2)\partial_3+s_3\partial_3^2-2\partial_1\partial_2-s_1\partial_2^2,
\end{align*} for $G^{(3)}$.
Therefore the conditions $\call_2(G^{(3)})=\call_3(G^{(3)})=0$ give the following recursions:
\begin{align*}
	(\kappa+b+2c-1)A_{a,b+1,c} & =A_{a+2,b,c},                   \\
	(\kappa+c-2)A_{a,b,c+1}    & =2A_{a+1,b+1,c}+aA_{a-1,b+2,c}.
\end{align*}
From these two equations, we have $(\kappa+c-2)A_{a,b,c+1}=(a+2\kappa+2b+4c)A_{a-1,b+2,c}$.
Then, we can find easily the solution of the recursions with the initial value $A_{a,0,0}=(\kappa-1)^{(a)}$.
\begin{align*}
	A_{a,b,c} & =\frac{(a+2b+3c+2\kappa-3)^{(c)}}{(\kappa-2)^{(c)}}(b+2c+\kappa-1)^{(a+b+c)}                               \\
	          & =\frac{(a+b+\kappa+\delta-1)^{(c)}}{(\kappa-2)^{(c)}}(b+\delta+1)^{(a)}(\delta+1)^{(b)}(\delta-c+1)^{(c)},
\end{align*}
where $\delta=b+3c+\kappa-2$.
Thus,
\begin{equation}\label{eq:G^3}
	G^{(3)}=\sum_{a,b,c\geq0}\binom{a+b+\delta}{a}\binom{b+\delta}{b}\binom{\delta}{c} \frac{(a+b+\kappa+\delta-1)^{(c)}}{(\kappa-2)^{(c)}}s_1^as_2^bs_3^c
\end{equation}
is a generating function.
The descending basis for $d = 3$ with a small degree is enumerated in Appendix A.

\section{Differential operators on Hermitian modular forms}
\subsection{Construction of the differential operators}
From now on, $\widetilde{G^{(\bfn)}}$ will also be denoted as ${G^{(\bfn)}}$ if there is no confusion.

We will construct the differential operators on $\calp^{\bfn}(\kappa)$ from the generating function $G^{(\bfn)}$.
We have $G^{(\bfn)}({}^tAXB,T)=G^{(\bfn)}(X,AT{}^tB)$ for any $A, B\in\gn$,
since $\sigma_i(({}^tAXB){}^tT)=\sigma_i(X {}^t(AT{}^tB))$.
We define the representations $(\rho_U, \bbc[[X]])$ and $(\rho'_U, \bbc[[X]])$ of $\gn\times\gn$ by
\[\rho_U(A,B)F(X)=F({}^tAXB) \quad \text{and} \quad \rho'_U(A,B)F(X)=F(A^{-1}X{}^tB^{-1})\]
for $A, B\in\gn$ and $F\in\bbc[[X]]$.
When $\gn\times\gn$ acts on $\bbc[[X]]$ by $\rho_U'$,
write $\bbc[[X]]$ as $(\bbc[[X]])'$ to avoid confusion.
Then, $G^{(\bfn)}$ can be regarded as an element of $(\calp^\bfn(\kappa)\otimes (\bbc[[X]])')^{\gn\times\gn}$.
Thus, for an representation $(\rho,V)$ of ${\gn\times\gn}$, if we send $G^{(\bfn)}$
with an element of $\mathrm{Hom}_{\gn\times\gn}(\bbc[[X]],V)$,
we obtain the element of $\calp^\bfn_\rho(\kappa)=(\calp^\bfn(\kappa)\otimes V')^{\gn\times\gn}$.
As a matter of fact, it yields an isomorphism.

\begin{prop}\label{prop:calpisom}
	We have an isomorphism
	\[
		\begin{array}{rccc}
			 & \mathrm{Hom}_{\gn\times\gn}(\bbc[[X]],V) & \quad\cong\quad & \calp^\bfn_\rho(\kappa), \\[5pt]
			 & c                                        & \mapsto         & c(G^{(\bfn)}).
		\end{array}
	\]
\end{prop}
\begin{proof}
	It can be proved as well as \cite[Theorem~3.1 (iii)]{Ibukiyama2020generic}.
\end{proof}

\begin{rem}
	As mentions in \cite[Remark 3.3]{Ibukiyama2020generic},
	the element of $\mathrm{Hom}_{\gn\times\gn}(\bbc[[X]],V)$ can be written down explicitly
	using the Capelli identity.
\end{rem}

Let $\bbc[X_{m,n}]$ be the space of polynomials
in the entries of $(m,n)$-matrices $X_{m,n}=(x_{ij})$ over $\bbc$.
We recall the irreducible decomposition of the polynomial ring $\bbc[X_{m,n}]$.
$\mathrm{GL}_m(\bbc)\times\mathrm{GL}_n(\bbc)$ acts on $\bbc[X_{m,n}]$ by
$(A,B)\cdot f(X_{m,n})=f({}^tAX_{m,n}B)$.
Then, we have the irreducible decomposition
\[\bbc[X_{m,n}]=\sum_{\ell(\bfk)\leq\min\left\{m,n\right\}}\rho_{m,\bfk}\boxtimes\rho_{n,\bfk}\]
(See, for example, \cite[Theorem 5.6.7]{Goodman2009symmetry}).
If we put \[\xi_i=\det\begin{pmatrix}
		x_{n-i+1,n-i+1} & \dots  & x_{1,n-i+1} \\
		\vdots          & \ddots & \vdots      \\
		x_{n-i+1,1}     & \dots  & x_{n,n}
	\end{pmatrix} \]
for $1\leq i\leq \min\left\{m,n\right\}$, then the highest weight vector
which generates the subspace of $\bbc[X_{m,n}]$
on which $\mathrm{GL}_m(\bbc)\times\mathrm{GL}_n(\bbc)$ acts by $\rho_{m,\bfk}\boxtimes\rho_{n,\bfk}$
with a dominant integral weight $\bfk=(k_1,\ldots,k_d,0,\ldots)$ is
\begin{equation}\label{eq:highestweight}
	\xi_{\bfk}:=\xi_1^{k_1-k_2}\xi_2^{k_2-k_3}\cdots\xi_{d-1}^{a_{d-1}-a_d}\xi_d^{a_d}.
\end{equation}
\begin{rem}
	The method of constructing a representation of $\mathrm{GL}_n(\bbc)$ with depth $m$ by taking a subrepresentation of the right regular representation on $\bbc[X_{m,n}]$ using the above irreducible decomposition is often called
	``a realization of a representation of $\mathrm{GL}_n(\bbc)$ by bideterminants''.
	We write the space of such representation $\rho$ as $\bbc[X_{m,n}]_\rho$
	and put $R[X_{m,n}]_\rho=\bbc[X_{m,n}]_\rho\cap R[X_{m,n}]$ for a subring $R$ of $\bbc$.
\end{rem}

\subsection{The case of $d=2$}
We will now discuss the case $d=2$, which can be calculated concretely.
We put $\bfn=(n_1,n_2)$ with $n=n_1+n_2$ and $n_1\geq n_2$.
We denote
$X=\begin{pmatrix}	0 &X_{1,2}\\X_{2,1}& 0 \end{pmatrix}$, where $X_{1,2}$ (resp. $X_{2,1}$)
is $(n_1, n_2)$-matrix (resp. $(n_2,n_1)$-matrix) of variables.
By the above, the representation $(\gn\times\gn, \bbc[X])$ has the irreducible decomposition
\[\bbc[X]=\sum_{\ell(\bfk),\ell(\bfl)\leq n_2}\rho_{n_1,(\bfk,\bfl)}\boxtimes\rho_{n_2,(\bfl,\bfk)}.\]
Combining this with Proposition~\ref{prop:calpisom}, we obtain the following proposition.
(For another more direct proof, see \cite[Proposition~17]{Takeda2025pullback}.)

\begin{prop}
	Let $(\bfk_i, \bfl_i)$ be a pair of dominant integral weights with $\ell(\bfk_i), \ell(\bfl_i)\leq n_2$ for $i=1,2$.
	We assume that $\ell(\bfk_i)+\ell(\bfl_i)\leq\kappa$ for $i=1,2$.
	When $d=2$, there exists a differential operator $\mathbb{D}$ satisfying the condition (A)
	for ${\det}^\kappa$ and ${\det}^\kappa\rho_{n_1,(\bfk_{1},\bfl_1)}\boxtimes{\det}^\kappa\rho_{n_2,(\bfk_{2},\bfl_2)}$
	if and only if $\bfk_1=\bfl_2$, $\bfl_1=\bfk_2$.
	If it exists, it is unique up to scalar multiplication.
\end{prop}

Consider the simplest case $n_1=n-1, n_2=1$.
If we put \[X_{1,2}={}^t\begin{pmatrix}x_{1,n}&x_{2,n}&\cdots & x_{n-1,n}\end{pmatrix}\]
and \[X_{2,1}=\begin{pmatrix}x_{n,1}&x_{n,2}&\cdots&x_{n,n-1}\end{pmatrix},\]
we have
\[X{}^tT=\begin{pmatrix}
		x_{1,n}t_{1,n}                           & \cdots            & x_{1,n}t_{n,n}                           \\
		\vdots                                   & \ddots            & \vdots                                   \\
		x_{n-1,n}t_{1,n}                         & \cdots            & x_{n-1,n}t_{n,n}                         \\
		x_{n,1}t_{1,1}+\cdots+x_{n,n-1}t_{1,n-1} & \quad \cdots\quad & x_{n,1}t_{n,1}+\cdots+x_{n,n-1}t_{n,n-1}
	\end{pmatrix}\]
and this is conjugate to
\[\begin{pmatrix}
		0      &        &   & \cdots                                                               & 0                                           \\
		\vdots &        &   & \ddots                                                               & \vdots                                      \\
		0      &        &   & \cdots                                                               & 0                                           \\
		*      & \cdots & * & \displaystyle\sum_{i=1}^{n-1}x_{i,n}t_{i,n}                          & x_{n-1,n}t_{n,n}                            \\
		*      & \cdots & * & \displaystyle\quad\frac{1}{x_{n-1,n}}\sum_{i=1}^{n-1}x_{i,n}a_i\quad & \displaystyle\sum_{i=1}^{n-1}x_{n,i}t_{n,i}
	\end{pmatrix},\]
where $a_i=\sum_{j=1}^{n-1}x_{n,j}t_{ij}$.
Thus we have
\begin{align*}
	\sigma_1 & =\sum_{i=1}^{n-1}(x_{i,n}t_{i,n}+x_{n,i}t_{n,i}),                                                                  \\
	\sigma_2 & =\sum_{1\leq i,\: j\leq n-1} x_{i,n}x_{n,j}t_{i,n}t_{n,j}-\sum_{1\leq i,\: j\leq n-1} x_{i,n}x_{n,j}t_{ij}t_{n,n},
\end{align*}
and $\sigma_i=0$ for $\geq3$.
Then the generating function $G^{(\bfn)}$ for this $\bfn$ is given by
\begin{align*}
	G^{(\bfn)} & =\left\{(1-\sigma_1/2)^2-\sigma_2\right\}^{-(\kappa-3/2)}                                           \\
	           & =\left\{1-\sum_{i=1}^{n-1}(x_{i,n}t_{i,n}+x_{n,i}t_{n,i})
	+\frac{1}{4}\left(\sum_{i=1}^{n-1}(x_{i,n}t_{i,n}-x_{n,i}t_{n,i})\right)^2\right.                                \\
	           & \hspace{10mm}\left.+\sum_{1\leq i,\: j\leq n-1} x_{i,n}x_{n,j}t_{ij}t_{n,n}\right\}^{-(\kappa-3/2)}
\end{align*}
Expanding this as a formal power series of $x_{i,n}$ and $x_{n,i}$ for all $i$,
the degree $k$ for $x_{i,n}$ and degree $l$ for $x_{n,i}$ term $P_{kl}(T)$ gives a
differential operator $\mathbb{D}_{kl}=P_{kl}(T)$ satisfies the condition (A)
for ${\det}^\kappa$ and ${\det}^\kappa\rho_{n_1,(\bfk,\bfl)}\boxtimes{\det}^\kappa\rho_{n_2,(\bfl,\bfk)}$ with $\bfk=(k,0,0,\ldots)$ and $\bfl=(l,0,0,\ldots)$,
which is equal to $({\det}^{\kappa+k}\boxtimes {\det}^l)\boxtimes({\det}^\kappa Sym^l\boxtimes Sym^k)$.

For the general $\bfn=(n_1,n_2)$, we give the generating function of the differential operators $\{\mathbb{D}_{kl}\}$, where $\mathbb{D}_{kl}$ satisfies
the condition (A) for ${\det}^\kappa$ and $({\det}^\kappa Sym^k\boxtimes Sym^l)\boxtimes({\det}^\kappa Sym^l\boxtimes Sym^k)$.

\begin{thm}\label{thm:diffsym}
	We put
	\begin{align*}
		G^{(\bfn)}
		 & =\left\{1-\sum_{\substack{1\leq i \leq n_1                                                                                                                                      \\1\leq j \leq n_2}}(t_{i,n_1+j}u_{1,i}v_{2,j}+t_{n_1+j,i}v_{1,j}u_{2,i})\right.\\[6pt]
		 & \hspace{10mm} +\frac{1}{4}\left(\sum_{\substack{1\leq i \leq n_1                                                                                                                \\1\leq j \leq n_2}}(t_{i,n_1+j}u_{1,i}v_{2,j}-t_{n_1+j,i}v_{1,j}u_{2,i})\right)^2\\[6pt]
		 & \hspace{10mm} +\left.\left(\sum_{1\leq i,\: j\leq n_1}t_{ij}u_{1,i}u_{2,j}\right)\left(\sum_{1\leq i,\: j\leq n_2}t_{n_1+i,n_1+j}v_{1,i}v_{2,j}\right)\right\}^{-(\kappa-3/2)}.
	\end{align*}
	Expanding this generating functions as a series in $u_1, u_2, v_1,v_2$,
	and replacing each $t_{ij}$ by $\frac{\partial}{\partial z_{ij}}$, the homogeneous part of degree $k$ for both
	$u_1$ and $v_2$ and degree $l$ for both $v_1$ and $u_2$ gives a differential operator $\mathbb{D}_{kl}$
	satisfying the condition (A) for ${\det}^\kappa$ and $({\det}^\kappa Sym^k\boxtimes Sym^l)\boxtimes({\det}^\kappa Sym^l\boxtimes Sym^k)$.
	Here we are taking the representation space of $Sym^k\boxtimes Sym^l$ (resp. $Sym^l\boxtimes Sym^k$) as the space of polynomials
	in $u_1 = (u_{1,1}, \ldots , u_{1,n_1} )$ and $v_1 = (v_{1,1}, \ldots , v_{1,n_2} )$
	(resp. $u_2 = (u_{2,1}, \ldots , u_{2,n_1} )$ and  $v_2 = (v_{2,1}, \ldots , v_{2,n_2} )$)
	which are homogeneous of degree $k$ for
	$u_1$ (resp. $v_2$) and of degree $l$ for $v_1$ (resp. $u_2$).
\end{thm}
\begin{proof}
	We check the conditions of the Proposition~\ref{prop:diff}.
	The second condition is almost obvious from the definition of $G^{(\bfn)}$, so we only check the pluriharmonicity.
	Due to symmetry, we only show that $D_{ij}G^{(\bfn)}=0$ for $1\leq i,j\leq n_1$.
	To simplify the notations, we put
	\begin{align*}
		\tau_{11} & =\sum_{1\leq i,\: j\leq n_1}t_{ij}u_{1,i}u_{2,j}, & \tau_{22} & =\sum_{1\leq i,\: j\leq n_2}t_{n_1+i,n_1+j}v_{1,i}v_{2,j}, \\
		\tau_{12} & =\sum_{\substack{1\leq i \leq n_1                                                                                          \\1\leq j \leq n_2}}t_{i,n_1+j}u_{1,i}v_{2,j}, &\tau_{21} &=\sum_{\substack{1\leq i \leq n_1\\1\leq j \leq n_2}}t_{n_1+j,i}v_{1,j}u_{2,i},
	\end{align*}
	and $f= 1-\tau_{12}-\tau_{21}+\frac{1}{4}(\tau_{12}-\tau_{21})^2+\tau_{11}\tau_{22}$
	(Note that $G^{(\bfn)}=f^{-(\kappa-3/2)}$).
	By a direct calculation, we have
	\begin{align*}
		\partial_{ij}f      & =u_{1,i}u_{2,j}\tau_{22}                                         & (1\leq i,j \leq n_1),                 \\
		\partial_{i,n_1+j}f & = u_{1,i}v_{2,j}\left(-1+\frac{1}{2}(\tau_{12}-\tau_{21})\right) & (1\leq i \leq n_1, 1\leq j \leq n_2), \\
		\partial_{n_1+i,j}f & = v_{1,i}u_{2,j}\left(-1-\frac{1}{2}(\tau_{12}-\tau_{21})\right) & (1\leq i \leq n_2, 1\leq j \leq n_1).
	\end{align*}
	For $1\leq i,j\leq n_1$, we have
	\begin{align*}
		 & D_{ij}G^{(\bfn)}=\left(\kappa\partial_{ij}+\sum_{k,l=1}^nt_{kl}\partial_{i,l}\partial_{k,j}\right)f^{-(\kappa-3/2)}                                             \\
		 & =\kappa\left(\frac{3}{2}-\kappa\right)f^{-(\kappa-1/2)}u_{1,i}u_{2,j}\tau_{22}                                                                                  \\
		 & \hspace{5mm}+\sum_{1\leq i,\: j\leq n_1}t_{kl}\left(\frac{3}{2}-\kappa\right)\left(\frac{1}{2}-\kappa\right)f^{-(\kappa+1/2)}
		u_{1,i}u_{2,l}u_{1,k}u_{2,j}\tau_{22}^2                                                                                                                            \\
		 & \hspace{5mm}+\sum_{\substack{1\leq k \leq n_1                                                                                                                   \\1\leq l \leq n_2}}t_{k,n_1+l}\left(\frac{3}{2}-\kappa\right)\left(\frac{1}{2}-\kappa\right)f^{-(\kappa+1/2)}u_{1,k}u_{2,j}\tau_{22}
		\cdot\frac{1}{2}u_{1,i}v_{2,l}(\tau_{12}-\tau_{21}-2)                                                                                                              \\
		 & \hspace{5mm}-\sum_{\substack{1\leq k \leq n_2                                                                                                                   \\1\leq l \leq n_1}}t_{n_1+k,l}\left(\frac{3}{2}-\kappa\right)\left(\frac{1}{2}-\kappa\right)f^{-(\kappa+1/2)}u_{1,i}u_{2,l}\tau_{22}
		\cdot\frac{1}{2}v_{1,k}u_{2,j}(\tau_{12}-\tau_{21}+2)                                                                                                              \\
		 & \hspace{5mm}-\sum_{1\leq i,\: j\leq n_2}t_{n_1+k,n_1+l}\left(\frac{3}{2}-\kappa\right)\left(\frac{1}{2}-\kappa\right)f^{-(\kappa+1/2)}                          \\
		 & \hspace{7mm}\cdot\frac{1}{4}u_{1,i}v_{2,l}v_{1,k}u_{2,j}(\tau_{12}-\tau_{21}-2)(\tau_{12}-\tau_{21}+2)                                                          \\
		 & \hspace{5mm}-\sum_{1\leq i,\: j\leq n_2}t_{n_1+k,n_1+l}\left(\frac{3}{2}-\kappa\right)f^{-(\kappa-1/2)}
		\cdot\frac{1}{2}u_{1,i}v_{2,l}v_{1,k}u_{2,j}                                                                                                                       \\
		 & =\left(\frac{3}{2}-\kappa\right)\left(\frac{1}{2}-\kappa\right)f^{-(\kappa+1/2)}u_{1,i}u_{2,j}\tau_{22}                                                         \\
		 & \hspace{7mm}\cdot\left\{\tau_{11}\tau_{22}+\frac{1}{2}(\tau_{12}-\tau_{21})^2-\tau_{12}-\tau_{21}-\frac{1}{4}\left\{(\tau_{12}-\tau_{21})^2-4\right\}-f\right\} \\
		 & =\left(\frac{3}{2}-\kappa\right)\left(\frac{1}{2}-\kappa\right)f^{-(\kappa+1/2)}u_{1,i}u_{2,j}\tau_{22}
		\left(1-\tau_{12}-\tau_{21}+\frac{1}{4}(\tau_{12}-\tau_{21})^2+\tau_{11}\tau_{22}-f\right)                                                                         \\
		 & =0.
	\end{align*}
	Thus, the pluriharmonicity of $G^{(\bfn)}$ is proved.
\end{proof}

We will also explain another way to construct the differential operators for more general cases when $d=2$.
We set the functions
\begin{align*}
	\delta_g(Z)             & = \det(CZ+D),                                       \\
	\Delta_g(Z)             & =(CZ+D)^{-1}C,                                      \\
	\varrho_g(Z; \kappa ,s) & =\abs{\det(CZ+D)}^{\kappa-2s}\det(CZ+D)^{-\kappa }, \\
	\delta(g)               & =\delta(g,\bfi_n),                                  \\
	\Delta(g)               & =(C\bfi_n+D)^{-1}(C+D\bfi_n),                       \\
	\varrho(g; \kappa ,s)   & = \abs{\delta(g)}^{\kappa-2s}\delta(g)^{-\kappa }   \\
\end{align*}
for $g=\begin{pmatrix}A&B\\C&D\end{pmatrix} \in \rmu(n,n)$, $Z=(Z_{ij})\in \hus{n}$,
a positive integer $\kappa$ and a complex variable $s$.
By a direct calculation, we obtain the following formula.
\begin{lem}\label{lem:deri}
	We have
	\begin{align*}
		\frac{\partial}{\partial Z_{ij}}\delta_g(Z)             & =\delta_g(Z)\Delta_g(Z)_{ji},                   \\
		\frac{\partial}{\partial Z_{ij}}\delta_g (Z)^{-\kappa } & =-\kappa \delta_g(Z)^{-\kappa}\Delta_g(Z)_{ij}, \\
		\frac{\partial}{\partial Z_{ij}}\Delta_g(Z)_{st}        & =-\Delta_g(Z)_{si}\Delta_g(Z)_{jt}.
	\end{align*}
	Similarly, we have
	\begin{align*}
		\pi^+_{ij}\cdot\delta(g)            & =\delta(g)\Delta(g)_{ji},                  \\
		\pi^+_{ij}\cdot\delta (g)^{-\kappa} & =-\kappa\delta(g)^{-\kappa}\Delta(g)_{ij}, \\
		\pi^+_{ij}\cdot\Delta(g)_{st}       & =-\Delta(g)_{si}\Delta(g)_{jt}.
	\end{align*}
	In particular, we have
	\begin{align*}
		\frac{\partial}{\partial Z_{ij}}(\varrho_g(Z;\kappa,s) ) & =\varrho_g(Z;\kappa,s) \left(-\frac{\kappa}{2}-s\right)\Delta_g(Z)_{ji},      \\
		\pi^+_{ij}(\varrho(g_1,g_2; \kappa ,s) )                 & =\varrho(g_1,g_2; \kappa ,s)  \left(-\frac{\kappa}{2}-s\right)\Delta(g)_{ji}.
	\end{align*}
\end{lem}

We put $\partial_Z=\left(\frac{\partial}{\partial Z_{ij}}\right)$ and $\pi^+=\left(\pi^+_{ij}\right)$.
As a simple consequence, we obtain the following lemma.
\begin{lem}\label{lem:difpoly}
	\begin{enumerate}
		\item For a polynomial $P(T)\in\bbc[T]$ with a degree $n$ matrix $T$ of variables
		      and a positive integer $\kappa$, there is a polynomial $\phi_{\kappa/2+s}(P)(T)\in\bbc[T]$ such that
		      \begin{align*}
			      P(\partial_Z)\varrho_g(Z;\kappa,s) & =\varrho_g(Z;\kappa,s) \phi_{\kappa/2+s}(P)(\Delta_g(Z)), \\
			      P(\pi^+)\varrho(g;\kappa,s)        & =\varrho(g;\kappa,s)\phi_{\kappa/2+s}(P)(\Delta(g)).
		      \end{align*}

		\item The polynomial $\phi_{\kappa/2+s}(P)$ in (1) also satisfies
		      \begin{align*}
			      P(\partial_Z)\delta (g,Z)^{-(\kappa/2+s)} & =\delta (g,Z)^{-(\kappa/2+s)} \phi_{\kappa/2+s}(P)(\Delta_g(Z)), \\
			      P(\pi^+)\delta(g)^{-(\kappa/2+s)}         & =\delta(g)^{-(\kappa/2+s)}\phi_{\kappa/2+s}(P)(\Delta(g)).
		      \end{align*}
	\end{enumerate}

\end{lem}

Now we calculate $\phi_\kappa(P_{\nu,\kappa}^D)$ for the descending basis $\{P_{\nu,\kappa}^D\}$.
\begin{prop}
	Let $P(T)$ be a homogeneous polynomial (of degree $d$).
	Then, we have
	\[\phi_\kappa(D^{(\kappa)}_{ij}P)(T)=-\frac{\partial \phi_\kappa(P)}{\partial t_{ji}}(T).\]
\end{prop}
\begin{proof}
	This can be proved as \cite[Theorem~2 (i)]{Ibukiyama2022Differential}.
	That proof, however, is a bit redundant, so we will write the proof again.
	Therefore, it is sufficient to show the case $\kappa\in \bbc\backslash\bbz_{<n}$, $g=\begin{pmatrix}0&I_n\\-I_n&0\end{pmatrix}$ and $Z=\sqrt{-1}Y$ with a positive definite Hermitian matrix $Y$.
	If we put
	\[\call_\kappa(P)(Y)=\int_{{\herm{n}}_{>0}}\exp\left(-\frac{1}{2}\tr(YT)\right)\, {}^tP(T)\det(T)^{\kappa-n}dT,\]
	where ${}^tP(T)=P({}^tT)$
	and $dT$ denotes the Lebesgue measure on $\herm{n}$, restricted to the open cone ${\herm{n}}_{>0}$ of positive definite Hermitian matrices,
	for a homogeneous polynomial $P(T)\in \bbc[T]$ of degree $d$ and a positive definite Hermitian matrix $Y\in{\herm{n}}_{>0}$,
	then we have
	\begin{equation}\label{eq:phil}
		\det(Y)^{-\kappa}\phi_\kappa(P)(Y^{-1})
		=c_n(\kappa)\left(-2 \right)^{-d}\call_\kappa(P)(Y),
	\end{equation}
	where $c_{n}(\kappa)=2^{-n\kappa}\pi^{-\frac{n(n-1)}{2}}\left(\prod_{i=0}^{n-1}\Gamma(\kappa-i)\right)^{-1}$,by (3.1) and (3.2) in \cite{Takeda2025pullback}.
	We put $S=(s_{ij}):=T^{-1}$ and $A=\dfrac{Y}{2}+(n-\kappa)S$.
	Since we have
	\[
		\call_\kappa(\partial_{ij}P)(Y)
		=\int_{{\herm{n}}_{>0}}a_{ij}\exp\left(-\frac{1}{2}\tr(YT)\right)\, {}^tP(T)\det(T)^{\kappa-n}dT,
	\]
	and
	\begin{align*}
		 & \quad\sum_{k,l=1}^n\call_\kappa(\partial_{il}\partial_{kj}P)(Y)                                                                    \\
		 & =\sum_{k,l=1}^n\int_{{\herm{n}}_{>0}}{}^tP(T)\left(\delta_{ki}\partial_{jk}+\delta_{lj}\partial_{li}+t_{lk}
		\partial_{jk}\partial_{li}\right)\left(\exp\left(-\frac{1}{2}\tr(YT)\right) \det(T)^{\kappa-n}\right)dT                               \\
		 & =\int_{{\herm{n}}_{>0}}(2na_{ij}+(ATA)_{ij}+n(n-\kappa)s_{ij})\exp\left(-\frac{1}{2}\tr(YT)\right)\,{}^tP(T) \det(T)^{\kappa-n}dT,
	\end{align*}
	we obtain
	\begin{equation}\label{eq:calldif}
		\call_\kappa(D_{ij}P)(Y)
		=\int_{{\herm{n}}_{>0}}\left(-\frac{\kappa}{2}y_{ij}+\frac{1}{4}(YTY)_{ij}\right)
		\exp\left(-\frac{1}{2}\tr(YT)\right)\,{}^tP(T) \det(T)^{\kappa-n}dT.
	\end{equation}

	On the other hand, if we denote \(\dfrac{\partial}{\partial Y} = \left( \dfrac{\partial}{\partial y_{ij}} \right)\), then applying \(\dfrac{\partial}{\partial Y}\) to both sides of equation~\eqref{eq:phil} yields
	\begin{align*}
		 & \det(Y)^{-\kappa} \left( -\kappa\, {}^tY^{-1} \phi_\kappa(P)(Y^{-1}) + \frac{\partial \phi_\kappa(P)(Y^{-1})}{\partial Y} \right)                                                                     \\
		 & \qquad = c_n(\kappa) \left(-2 \right)^{-d} \frac{\partial \mathcal{L}_\kappa(P)(Y)}{\partial Y}                                                                                                       \\
		 & \qquad = \left( c_n(\kappa) \left(-2 \right)^{-d} \int_{{\herm{n}}_{>0}} \left( -\frac{t_{ji}}{2} \right) \exp\left( -\frac{1}{2} \tr(YT) \right)\, {}^tP(T) \det(T)^{\kappa - n} \, dT \right)_{ij}.
	\end{align*}
	Thus, we obtain
	\begin{align*}
		 & \det(Y)^{-\kappa} \left( -\kappa\, {}^tY \phi_\kappa(P)(Y^{-1}) + {}^tY \frac{\partial \phi_\kappa(P)(Y^{-1})}{\partial Y} {}^tY \right)                                                           \\
		 & = \left( c_n(\kappa) \left(-2 \right)^{-d} \int_{{\herm{n}}_{>0}} \left( -\frac{(YTY)_{ji}}{2} \right) \exp\left( -\frac{1}{2} \tr(YT) \right)\, {}^tP(T) \det(T)^{\kappa - n} \, dT \right)_{ij}.
	\end{align*}
	Combining this with equation~\eqref{eq:calldif}, and noting that
	\[
		{}^tY \frac{\partial \phi_\kappa(P)(Y^{-1})}{\partial Y} {}^tY = \frac{\partial \phi_\kappa(P)(Y^{-1})}{\partial Y^{-1}},
	\]
	and that \(D_{ij}P(T)\) is a homogeneous polynomial of degree \(d - 1\), we deduce that
	\[
		\phi_\kappa(D_{ij}P)(T) = -\frac{\partial \phi_\kappa(P)(T)}{\partial t_{ji}}
	\]
	by equation~\eqref{eq:phil}.
\end{proof}

Since $P_{0,\kappa}^D=1$, we inductively obtain the following corollary.
\begin{cor}If $\kappa\in \bbc\backslash\bbz_{<n}$, then we have
	\[\phi_\kappa(P_{\nu,\kappa}^D)(T)=(-1)^{|\nu|}\frac{({}^tT)^\nu}{\nu!}\]
	for any $\nu\in\caln_n$, where we put $\nu! = \prod_{i,j=1}^n\nu_{ij}!$ and $|\nu| = \sum_{i,j=1}^n\nu_{ij}$.
\end{cor}

Note that, more generally, the image of $\phi_\kappa$ can be written as follows.
\begin{prop}[\cite{Takeda2025pullback} Theorem~6]\label{thm:PtoPhi}
	We assume $\kappa\in \bbz_{>n}$.
	Then we have
	\[\phi_\kappa(P)(T)=\left.\left(-1\right)^{d}\left(P(\partial_W)\det(I_n-W\, ^tT)^{-\kappa}\right)\right|_{W=0}\]
	for a homogeneous polynomial $P(T)\in \bbc[T]$ of degree $d$.
\end{prop}

By the above corollary, we obtain a formula that gives $P(T)$ such that $\phi_\kappa(P)=Q$ for a homogeneous polynomial $Q(T)$ inversely.

\begin{thm}
	We assume $\kappa\in\bbc\backslash \bbz_{\leq n}$. Let $G^{(n)}(X,T)=\sum_{\nu\in\caln_n}P_{\nu}(X)T^\nu$ be the generating function for the descending basis of $\calp^{(n)}(\kappa)$ such that $P_\nu(X)=c_{|\nu|}P^D_{\nu,\kappa}(X)$ with some constants $c_{|\nu|}\in \bbc$.
	For a homogeneous polynomial $Q(T)\in\bbc[T]$ of degree $d$,
	we put
	\[P(T)=\left.\frac{(-1)^d}{c_d}Q(\partial_X)G^{(n)}(X,T)\right|_{X=0}.\]
	Then, we have $\phi_\kappa(P)=Q$.
\end{thm}
\begin{ex}
	For $G^{(n)}$ defined in \eqref{eq:generating} with $G^{(1)}=(1-s_1/2)^{3-2\kappa}$, $c_d$ in the above theorem is given by
	$c_d=(\kappa)^{(d)}(2\kappa-3)^{(d)}$.
\end{ex}

We fix $\kappa>n=n_1+n_2$.
Let $\bfk$ and $\bfl$ be a pair of dominant integral weights with $\ell(\bfk), \ell(\bfl)\leq n_2$ and $\ell(\bfk)+\ell(\bfl)\leq\kappa$.
We put $m=\max\{\ell(\bfk), \ell(\bfl)\}$.
For a $(2m,2m)$-matrix $S$ of variables, we write the $(m,m)$-matrix blocks $S_{ij}$ of $S$ as
\[S=\begin{pmatrix}
		S_{1,1} & S_{1,2} \\
		S_{2,1} & S_{2,2}
	\end{pmatrix}.\]
We define a polynomial $Q_0(S)\in\bbc[S]$ by $Q_0(S)=(S_{1,2})_{\bfk}(S_{2,1})_{\bfl}$.
Then there exists a unique polynomial $P_0(S)\in\bbc[S]$ such that $\phi_\kappa(P_0)=Q_0$.
Now taking $(m,n_1)$-matrices $U_1, U_2$ and $(m,n_2)$-matrices $V_1, V_2$,
we define $(2m, n)$-matrices $\mathbb{U}_1, \mathbb{U}_2$ as
\[\mathbb{U}_i=\begin{pmatrix}
		U_i & 0   \\
		0   & V_i
	\end{pmatrix}.\]
We regard the representation spaces of $\mathrm{GL}_{n_i}(\bbc)\times\mathrm{GL}_{n_i}(\bbc)$ as subspaces of $\bbc[U_i,V_i]$ via the bideterminant.

\begin{thm}\label{thm:diff2}
	The notation is as above, and for an $(n,n)$-matrix T of variables, we put
	\[P(T)=P_0(\mathbb{U}_1T\;{}^t\mathbb{U}_2).\]
	Then the differential operator $P(\partial_Z)$ satisfies condition (A) for ${\det}^\kappa$ and $\rho_{n_1,(\bfk,\bfl)}\boxtimes\rho_{n_2,(\bfl,\bfk)}$.
\end{thm}
\begin{proof}
	This can be proved as \cite[Theorem~4]{Ibukiyama2022Differential}.
\end{proof}
\begin{rem}
	We have $\phi_\kappa(P)=Q_0(\mathbb{U}_1T\;{}^t\mathbb{U}_2)$ for the above $P$ and $Q_0$.
\end{rem}

\section{Exact pullback formula}
We introduce the Hermitian Eisenstein series according to Shimura
\cite[\S 16.5]{shimura2000arithmeticity}. We fix a family $\kappa=(\kappa_v)_{\va }$ of positive integers and an integral ideal $\frakn$ of $\kp$.
The pullback formula recalled below is essentially obtained from double coset decompositions, together with the transformation properties of the differential operators.
In this section, the remaining task is to evaluate explicitly the archimedean factor appearing in the formula.
This explicit factor is important for arithmetic applications, including formulas for critical values of standard $L$-functions and the study of Fourier expansions of Hermitian Klingen-Eisenstein series.

Consider the following subgroups of $G_{n}$ for $r\leq n$:
\begin{align*}
	L_{n,r} & =\left\{\left.
	\begin{pmatrix}
		A & 0   & 0            & 0   \\
		0 & I_r & 0            & 0   \\
		0 & 0   & \adj{A}^{-1} & 0   \\
		0 & 0   & 0            & I_r \\\end{pmatrix}
	\in G_{n}\right| A \in \mathrm{GL}_{n-r}(K)\right\}, \\
	U_{n,r} & =\left\{
	\begin{pmatrix}
		I_{n-r} & *   & *       & *   \\
		0       & I_r & *       & 0   \\
		0       & 0   & I_{n-r} & 0   \\
		0       & 0   & *       & I_r \\\end{pmatrix}
	\in G_{n}\right\},                                   \\
	G_{n,r} & =\left\{
	\begin{pmatrix}
		I_{n-r} & 0 & 0       & 0 \\
		0       & * & 0       & * \\
		0       & 0 & I_{n-r} & 0 \\
		0       & * & 0       & * \\\end{pmatrix}
	\in G_{n}\right\}.
\end{align*}
Then the subgroups $P_{n,r}=G_{n,r}L_{n,r}U_{n,r}$ are the standard parabolic subgroups of $G_n$ and there are natural embeddings
$t_{n,r}: \mathrm{GL}_{n-r}(K)\hookrightarrow L_{n,r}$ and $s_{n,r}: G_{r}\hookrightarrow G_{n,r}$.
Define $G_{n,r,v}, G_{n,r, \adele}, L_{n,r,v}, L_{n,r, \adele}$, etc.\ in the same way as $G_{n,v}, G_{n, \adele}$, etc.

By the Iwasawa decomposition, $G_{n, \adele}$ (resp. $G_{n,v}$) can be decomposed as
$G_{n, \adele}=P_{n,r,\adele}K_{n,\adele}$ (resp. $G_{n,v}=P_{n,r,v}K_{n,v}$).

We take a Hecke character $\chi$ of $K$ whose infinite part $\chi_\infty$
satisfies
\begin{equation}\label{eq:chi_infty}
	\chi_\infty(x)=\prod_{\va} |x_v|^{\kappa_v}x_v^{-\kappa_v},
\end{equation}
where $x_\infty$ is the infinite part of $x$, and of the conductor dividing $\frakn$.
We put $\chi_v=\prod_{w|v}\chi_w$ for a place $v$ of $\kp$.

\begin{dfn}
	We define
	\[\epsilon_{n,\kappa,v}(g,s; \frakn, \chi)=\left\{
		\begin{array}{ll}
			\norm{\det(\adj{A}A)}_v^{s}\chi_v(\det A)           & (v\in\bfh \text{ and } k\in K_{n,v}(\frakn) ),     \\
			0                                                   & (v\in\bfh \text{ and } k\not\in K_{n,v}(\frakn) ), \\
			\norm{\delta(g)}^{\kappa_v-2s}\delta(g)^{-\kappa_v} & (\va )
		\end{array}
		\right. \]
	for a complex variable $s$ and $g=t_{n,0}(A)\mu k \in G_{n,v}$
	with $A\in \mathrm{GL}_n(K_v)$, $\mu\in U_{n,0}$ and $k\in K_{n,v}$.
	Then, we put
	\[\epsilon_{n,\kappa}(g,s;\frakn,\chi)=\prod_v\epsilon_{n,\kappa,v}(g_v,s;\frakn,\chi),\]
	and define the Hermitian Eisenstein series $E_{n,\kappa}(g,s;\frakn,\chi)$ on $G_{n, \adele}$ by
	\[E_{n,\kappa}(g,s;\frakn,\chi)=\sum_{\gamma\in P_{n,0}\backslash G_n}\epsilon_{n,\kappa}(\gamma g,s; \frakn, \chi).\]
	For $\theta\in K_{n,0}$, we put
	\[E_{n,\kappa}^\theta(g,s;\frakn,\chi)=E_{n,\kappa}(g\theta^{-1},s; \frakn,\chi).\]
\end{dfn}
The Hermitian Eisenstein series $E_{n,\kappa}(g,s;\frakn,\chi)$ and $E^\theta_{n,\kappa}(g,s;\frakn,\chi)$ converge absolutely and locally uniformly for $\Re(s)>n$
(see, for example, \cite{Shimura1997Euler}).
\begin{prop}[{\cite[Proposition~17.7]{shimura2000arithmeticity}}]
	Let $\mu$ be a positive integer such that $\mu\geq n$.
	If $\kappa_v=\mu$ for any $\va$, Then $E_{n,\kappa}(g,\mu/2;\frakn,\chi)$ belongs to $\mathcal{A}_n({\det}^\mu,\frakn)$
	except when $\mu=n+1$, $\kp=\bbq$, $\chi=\chi_K^{n+1}$,
	where $\chi_K$ is the quadratic character associated to quadratic extension $K/\kp$.
\end{prop}

Let $\rho_{r}$ be the representation of $K_{r,\infty}^\bbc$ corresponding to
a family $(k_{1,v},\ldots,k_{r,v};l_{1,v},\ldots,l_{r,v})_{\va}$ of dominant integral weights.
For $n\geq r$, we define the representation $\rho_{n}$ of $K_{n,\infty}^\bbc$ as the representation corresponding to a family $(k_{1,v},\ldots,k_{r,v},k_{r,v},\ldots,k_{r,v}; l_{1,v},\ldots,l_{r,v},0,\ldots,0)_{\va}$ of dominant integral weights.

\begin{dfn}
	We define
	\[\epsilon(f)_{r,v}^n(g,s ;\chi)=\left\{
		\begin{array}{ll}
			\abs{\det \left(\adj{A_r}A_r\right)}_v^{s}\chi_v(\det A_r)f(h_r)                    & \hspace{-20mm} (v\in\bfh \text{ and } k\in K_{n,v}(\frakn)  ),    \\
			0                                                                                   & \hspace{-20mm} (v\in\bfh \text{ and } k\not\in K_{n,v}(\frakn) ), \\
			\norm{\delta(g)\delta(h_r)^{-1}}^{\kappa_v-2s}\rho_n(M(g))^{-1}\rho_r(M(h_r))f(h_r) & (\va )
		\end{array}
		\right.\]
	for $f\in \mathcal{A}_{0,r}(\rho_r,\frakn)$ ($r<n$)
	and $g=t_{n,r}(A_r)\, \mu_r\,	s_{n,r}(h_r)\, k \in G_{n,v}$ with $A_r\in \mathrm{GL}_{n-r}(K_v)$, $\mu_r\in U_{n,r}$,
	$h_r\in G_{r,v}$ and $k\in K_{n,v}$.
	Then, we put
	\[\epsilon(f)_{r}^n(g,s;\chi)=\prod_v\epsilon(f)_{r,v}^n(g_v,s ;\chi)\]
	and define the Hermitian Klingen-Eisenstein series $\left[f\right]_{r}^n(g,s;\chi)$ on
	$G_{n, \adele}$ associated with $f$ by
	\[\left[f\right]_r^n(g,s;\chi)=\sum_{\gamma\in P_{n,r}\backslash G_n}\epsilon(f)_{r}^n(\gamma g,s;\chi).\]
\end{dfn}

Let $n_1, n_2$ be positive integers such that $n_1\geq n_2$,
$m=m_{\kp}$ be the number of infinite places of $\kp$,
$\kappa=(\kappa_v)_{\va }$ a family of positive integers and $\bfk=(\bfk_v)_{\va }$ and $\bfl=(\bfl_v)_{\va }$
a family of dominant integral weights such that $\ell(\bfk_v)< n_2$, $\ell(\bfl_v)\leq n_2$
and $\ell(\bfk_v)+\ell(\bfl_v)\leq\kappa_v$ for each infinite place $v$ of $\kp$.
We put $\rho_{n_1}={\det}^{\kappa}\rho_{n_1,(\bfl,\bfk)}$, $\rho_{n_2}={\det}^{\kappa}\rho_{n_2,(\bfk,\bfl)}$ and let $\dkl=P(\partial_Z)=\prod_{\va}P_v(\partial_{Z_v})$ be the differential operator satisfying the condition (A) for ${\det}^{\kappa}$ and $\rho_{n_1}\boxtimes\rho_{n_2}$, which is defined in Theorem~\ref{thm:diff2}.

We consider the representation space $V_{\rho_{{n_1},(\bfl,\bfk)}}$ (resp. $V_{\rho_{{n_2},(\bfk,\bfl)}}$) as the subspace of $\bbc[V_1,U_1]$ (resp. $\bbc[U_2,V_2]$),
where $U_1=(U_{1,v})_{\va}$ and $V_1=(V_{1,v})_{\va}$ (resp. $U_2=(U_{2,v})_{\va}$ and $V_2=(V_{2,v})_{\va}$) are families of $(n_1,n_1)$-matrices (resp. $(n_2,n_2)$-matrices) of variables.

Assuming $\bbc[U_2,V_2]\subset\bbc[V_1,U_1]$, we often consider $V_{\rho_{n_2,(\bfk,\bfl)}}$ as a subspace of $V_{\rho_{{n_1},(\bfl,\bfk)}}$.
In order to avoid an ambiguity depending on an inner product of $V_{\rho_{n_2,(\bfk_v,\bfl_v)}}\subset \bbc[U_2,V_2]$, we fix it.
For polynomials $P(U_2,V_2), Q(U_2,V_2)\in V_{\rho_{n_2,(\bfk,\bfl)}}$, we define the inner product of $V_{\rho_{n_2,(\bfk,\bfl)}}$ by
\[\left<P(U_2,V_2),Q(U_2,V_2)\right>=P(\partial_{U_2},\partial_{V_2})\overline{Q}(U_2,V_2)\mid_{U_2=V_2=0},\]
where $\overline{Q}$ is the polynomial obtained by replacing the coefficients by the complex conjugation.

Let $\frakn$ be an integral ideal
and take a Hecke character $\chi$ of $K$ whose infinite part $\chi_\infty$ satisfies
\[\chi_\infty(x)=\prod_{\va} |x_v|^{\kappa_v}x_v^{-\kappa_v},\]
and whose conductor divides $\frakn$.

Let $\chi_{K}$ be the quadratic character associated to quadratic extension $K/\kp$.
For a Hecke eigenform $f$ on $G_{n,\adele}$ of level $\frakn$, and weight $(\rho,V)$
and a Hecke character $\eta$ of $K$, we set
\[D(s,f; \eta)
	=L(s-n+1/2,f\otimes \eta,\mathrm{St})
	\cdot\left(\prod_{i=0}^{2n-1} L_{\kp}(2s-i,\eta\cdot\chi_{K}^{i})\right)^{-1},
\]
where $L(*,f\otimes {\eta}, \mathrm{St})$ is the degree $4n$ standard L-function attached
to $f\otimes {\eta}$ and $L_{\kp}(*,\eta)$ (resp. $L_{\kp}(*,\eta\cdot\chi_{K})$) is the Hecke L-function
attached to $\eta$ (resp. $\eta\cdot\chi_{K}$).
For a finite set $S$ of finite places, we put
\[D_S(s,f;\eta)=\prod_{v\in\bfh\backslash S}D_v(s,f;\eta),\]
where $D_v(s,f;\eta)$ is a $v$-part of $D(s,f; \eta)$.

We put $g^\natural=\left(\begin{smallmatrix}0&I_r\\I_r&0\\ \end{smallmatrix}\right)
	g\left(\begin{smallmatrix}0&I_r\\I_r&0\\ \end{smallmatrix}\right)$ and define $f^\dagger_\chi$ by
\[
	f_\chi^\dagger(g)
	= \chi(\det g)\,\overline{f(g^\natural)},
\]
where $\chi$ is a Hecke character of $K$ whose archimedean component
$\chi_\infty$ satisfies the above condition.

From Lemma~\ref{lem:difpoly}, there is a polynomial $\phi_\kappa(P)(T,s)=\prod_{\va}\phi_{\kappa_v/2+s}(P_v)(T)$ such that
\[\dkl\epsilon_{n,\kappa,\infty}(g,s; \frakn,\chi)=\epsilon_{n,\kappa,\infty}(g,s)\phi_\kappa(P)(\Delta(g),s).\]

We put $|\bfk|=\sum_{v,i}k_{v,i}$, $|\bfl|=\sum_{v,i}l_{v,i}$
for the fixed dominant integral weights such that $\bfk_v=(k_{v,1},k_{v,2},\ldots)$, $\bfl_v=(l_{v,1},l_{v,2},\ldots)$ for each $\va$, and $|\kappa|=n_2\sum_{\va}\kappa_v$.

\begin{thm}[Theorem~28 in \cite{Takeda2025pullback}]\label{thm:pullback}
	Let $S$ be the finite set of finite places dividing $\frakn$,
	and we take $s \in \bbc$ such that $\Re(s)>n$.
	\begin{enumerate}
		\item If $n_1=n_2$, for any Hecke eigenform $f\in \mathcal{A}_{0,n_2}(\rho_{n_2},\frakn)$, we have
		      \begin{align*}
			      \left(f,(\dkl E^\theta_{n,\kappa})(\iota(g_1,*),\overline{s}; \chi)\right)
			      =  c(s,\rho_{n_2}) \cdot \prod_{v\mid\frakn}[\calk_{n,v}:\calk_{n,v}(\frakn)]
			      \cdot D_S(s,f;\overline{\chi}) \cdot \overline{f_\chi^\dagger(g_1)}.
		      \end{align*}
		\item If $\frakn=\inte{\kp}$, for any Hecke eigenform $f\in \mathcal{A}_{0,n_2}(\rho_{n_2})$, we have
		      \begin{align*}
			      \left(f,(\dkl E_{n,\kappa})(\iota(g_1,*),\overline{s};\frakn, \chi)\right)
			      =  c(s,\rho_{n_2}) \cdot D(s,f;\overline{\chi}) \cdot \overline{[f_\chi^\dagger]_{n_2}^{n_1}(g_1,s; \chi)}.
		      \end{align*}
	\end{enumerate}
	Here, for any $w\in V_{\rho_{n_2}}$, the function $c(s,\rho_{n_2})$ satisfies
	\begin{align*}
		c(s,\rho_{n_2})w=2^{-mn_2(2s+2n_2)+2|\kappa|+|\bfk|+|\bfl|} & \int_{\mathfrak{S}^\bfa_{n_2}}\left<\rho_{n_2}(I_{n_2}-\adj{S}S,I_{n_2}-{}^{t}\!{S}\overline{S})w,
		\phi_\kappa(P)(R,\overline{s})\right>                                                                                                                            \\
		                                                            & \qquad\cdot\det(I_{n_2}-\adj{S}S)^{\kappa/2-s-2n_2}dS,
	\end{align*}
	where $\mathfrak{S}_{n_2}=\left\{S\in
		M_{n_2}(\bbc)\mid I_{n_2}-\adj{S}S>0\right\}$, $dS$ is the Lebesgue measure on $\mathfrak{S}_{n_2}$, and
	\begin{align*}
		R= &
		\begin{pmatrix}
			\begin{pmatrix}0&0\\0&\sqrt{-1}S\end{pmatrix}
			 & \begin{pmatrix}0\\I_{n_2}\end{pmatrix}   \\
			\begin{pmatrix}0&I_{n_2}\end{pmatrix}
			 & -\sqrt{-1}\adj{S}(I_{n_2}-S\adj{S})^{-1}
		\end{pmatrix}.
	\end{align*}
\end{thm}
In this theorem, only the value of $c(s,\rho_{n_2})$ is not specifically known.
In the following, we will calculate it for a particular case using the above explicit construction of the differential operators.
Note that $\phi_\kappa(P)(R,\overline{s})$ in the theorem does not depend on $S$ from Theorem~\ref{thm:diff2}
and the integral defining $c(s,\rho_{n_2})$ can be decomposed and computed for each place $\va$.

\begin{lem}
	If we put
	\[\mathfrak{S}_{m,n}=\left\{S\in M_{m,n}(\bbc)\mid I_{n}-\adj{S}S>0\right\}\]
	and
	\[I_{m,n}(s,\bfk,\bfl)=\int_{\mathfrak{S}_{m,n}}\det(I_{n}-\adj{S}S)^{s}
		\rho_{n,(\bfk,\bfl)}(I_{n}-\adj{S}S,I_{n_2}-{}^{t}\!{S}\overline{S}) dS\]
	for a representation $(\rho_{n,(\bfk,\bfl)},V_{(\bfk,\bfl)})$ of $\mathrm{GL}_n(\bbc)\times\mathrm{GL}_n(\bbc)$
	with dominant integral weights $\bfk=(k_1,\ldots,k_n)$, $\bfl=(l_1,\ldots,l_n)$,
	we have
	\[I_{m,n}(s,\bfk,\bfl)=c_{m,n}(s,\bfk,\bfl)\cdot \mathrm{id}_{V_{(\bfk,\bfl)}},\]
	where
	\[c_{m,n}(s,\bfk,\bfl)=\pi^{nm}\prod_{i=1}^n \frac{\Gamma(s+k_i+l_i+n-i+1)}{\Gamma(s+k_i+l_i+n+m-i+1)}
		=\pi^{nm}\prod_{i=1}^n \frac{1}{(s+k_i+l_i+n+m-i)_{(m)}}\]
	where $(x)_{(r)}=x(x-1)\cdots(x-r+1)$ is a descending Pochhammer symbol.
\end{lem}
\begin{proof}
	By Schur's Lemma, $I_{m,n}(s,\bfk,\bfl)$ is a scalar function.
	We calculate $c_{m,n}(s,\bfk,\bfl)$ in the same way as \cite[\S2.3]{Hua1963harmonic} and \cite[Lemma 2]{Kozima2002Standard}.
	Let $\omega_0\in V_{(\bfk,\bfl)}$ be the highest weight vector with $\left<\omega_0,\omega_0\right>=1$.
	If we put $\bfk_0=(k_1-k_n,\ldots,k_{n-1}-k_n,0)$ and $\bfl_0=(l_1-l_n,\ldots,l_{n-1}-l_n,0)$, then we have
	\[c_{m,n}(s,\bfk,\bfl)=\int_{\mathfrak{S}_{m,n}}\det(I_{n}-\adj{S}S)^{s+k_n+l_n}
		\left<\rho_{n,(\bfk_0,\bfl_0)}(I_{n}-\adj{S}S,I_{n_2}-{}^{t}\!{S}\overline{S})\omega_0,\omega_0\right> dS.\]
	If we set
	\[S=\begin{pmatrix}	S_1 &q\end{pmatrix},\]
	where $S_1\in M_{m,n-1}(\bbc)$ and $q\in M_{m,1}$
	then $I_{n-1}-\adj{S_1}S_1>0$ and there exists a nonsingular matrix $\Gamma$ such that
	\[\Gamma^*\Gamma=I_{n-1}-\adj{S_1}S_1.\]
	We make the substitution $q=\Gamma w_n$ and put
	\[\rho^0_{n,(\bfk_0,\bfl_0)}(g)=\rho_{n,(\bfk_0,\bfl_0)}\left(\begin{pmatrix}g_1&0\\0&1\end{pmatrix},\begin{pmatrix}g_2&0\\0&1\end{pmatrix}\right)\]
	for $g_1,g_2\in \mathrm{GL}_{n-1}(\bbc)$.
	Then we have
	\begin{align*}
		c_{m,n}(s,\bfk)
		 & =\int_{1-w_n^*w_n>0}(\det(1-w_n^*w_n))^{s+k_n+l_n}dw_n                         \\
		 & \qquad\cdot\int_{\mathfrak{S}_{m,n-1}}\det(I_{n-1}-\adj{S_1}S_1)^{s+k_n+l_n+1}
		\left<\rho^0_{n,(\bfk_0,\bfl_0)}(I_{n}-\adj{S_1}S_1)\omega_0,\omega_0\right>  dS_1.
	\end{align*}
	Repeat this operation, we have
	\begin{align*}
		c_{m,n}(s,\bfk) & =\prod_{i=1}^n\int_{w_i^*w_i<1}(\det(1-w_i^*w_i))^{s+k_i+l_i+n-i}dw_i\cdot\left<\omega_0,\omega_0\right> \\
		                & =\prod_{i=1}^n \pi^m\frac{1}{(s+k_i+l_i+n+m-i)_{(m)}}.
	\end{align*}
\end{proof}

Let $v_{0,v}=(U_{2,v})_{\bfk_{2,v}} (V_{2,v})_{\bfl_v}\in V_{\rho_{n_2,(\bfk_v,\bfl_v)}}$ be the highest weight vector of $\rho_{n_2,(\bfk_v,\bfl_v)}$ defined as in \eqref{eq:highestweight}
and $v_{0,v}^*=(U_{1,v})_{\bfl_v} (V_{1,v})_{\bfk_v}$ the highest weight vector of $\rho_{n_1,(\bfl_v,\bfk_v)}$ (Then we may consider $v_{0,v}=v_{0,v}^*$ under proper identification.).
For the same reason in \cite[Lemma 4 and Proposition 3]{Ibukiyama2022Differential}, we have
\begin{align*}
	\left<v_{0,v}, \phi_{\kappa_v}(P_v)(R)\right> & =\left<v_{0,v},v_{0,v}\right>v_{0,v}^*                                   \\
	                                              & =\frac{\left(\prod_{i=1}^{\ell(\bfk_v)}(k_{v,i}+\ell(\bfk_v)-i)!\right)}
	{\left(\prod_{1\leq i < j\leq\ell(\bfk_v)}(k_{v,i}-k_{v,j}+j-i)\right)}\cdot
	\frac{\left(\prod_{i=1}^{\ell(\bfl_v)}(l_{v,i}+\ell(\bfl_v)-i)!\right)}
	{\left(\prod_{1\leq i < j\leq\ell(\bfl_v)}(l_{v,i}-l_{v,j}+j-i)\right)}v_{0,v}^*
\end{align*}

From the above, we obtain the following results.
\begin{thm}\label{thm:c} Let $\mu$ be a positive integer such that $\mu\geq n$.
	Assume that $\kappa_v=\mu$ for all $\va$.
	Then, $c(\mu/2,\rho_{n_2})$ in Theorem~\ref{thm:pullback} is given by
	\begin{align*}
		c(\mu/2,\rho_{n_2})= & 2^{-2mn_2^2-m(n_2-2)\mu+|\bfk|+|\bfl|}\pi^{mn_2^2}                                          \\
		                     & \quad\cdot\prod_{\va}\left(\frac{\prod_{i=1}^{\ell(\bfk_v)}(k_{v,i}+\ell(\bfk_v)-i)!}
		{\prod_{1\leq i < j\leq\ell(\bfk_v)}(k_{v,i}-k_{v,j}+j-i)}\cdot
		\frac{\prod_{i=1}^{\ell(\bfl_v)}(l_{v,i}+\ell(\bfl_v)-i)!}
		{\prod_{1\leq i < j\leq\ell(\bfl_v)}(l_{v,i}-l_{v,j}+j-i)}\right.                                                  \\
		                     & \hspace{15mm}\cdot\left.\prod_{i=1}^{n_2} \frac{1}{(k_{v,i}+l_{v,i}+\mu-i)_{(n_2)}}\right),
	\end{align*}
	where $(x)_{(r)}=x(x-1)\ldots(x-r+1)$ is a descending Pochhammer symbol.
\end{thm}
If we put
\[c(s,f)=c(s,\rho_{n_2}) \cdot D(s,f;\overline{\chi})\]
for a Hecke eigenform $f\in \mathcal{A}_{0,n_2}(\rho_{n_2},\frakn)$, then we have the following corollary.
\begin{cor}
	For a Hecke eigenform $f\in \mathcal{A}_{0,n_2}(\rho_{n_2},\frakn)$,
	the function $c(s,f)\cdot [f]_{n_2}^{n_1}(g_1,s; \chi)$ for $s$ has meromorphic continuation to the whole complex plane.
\end{cor}
\begin{cor}
	Assume that $c(s,\rho_{n_2})$ is a meromorphic function in $s$ and it is not identically zero.
	Then, for a Hecke eigenform $f\in \mathcal{A}_{0,n_2}(\rho_{n_2},\frakn)$ and a Hecke character $\chi$ of $K$ whose infinite part $\chi_\infty$ satisfies \eqref{eq:chi_infty},
	the function $[f]_{n_2}^{n_1}(g_1,s; \chi)$ for $s$ has meromorphic continuation to the whole complex plane.
\end{cor}
\newpage
\appendix
\section{The descending basis for $d=3$}
We give a descending basis for $d=3$ with small degree $|\nu|\leq 4$.
For $X=	\left(\begin{smallmatrix}0&x_{12}&x_{13}\\x_{21}&0&x_{23}\\x_{31}&x_{32}&0\end{smallmatrix}\right)$ and
$T=	\left(\begin{smallmatrix}t_{11}&t_{12}&t_{13}\\t_{21}&t_{22}&x_{23}\\t_{31}&t_{32}&t_{33}\end{smallmatrix}\right)$,
$\sigma_i$ in \eqref{eq:sigma} is given by
\begin{align*}
	\sigma_0 & =1,                                          \\
	\sigma_1 & = t_{12}x_{12} + t_{13}x_{13} + t_{21}x_{21}
	+ t_{23}x_{23} + t_{31}x_{31} + t_{32}x_{32},           \\
	\sigma_2 & =
	(t_{21}t_{12}-t_{22}t_{11})\,x_{12}x_{21}
	+(t_{12}t_{23} - t_{22}t_{13})\,x_{12}x_{23}
	+(t_{31}t_{12}-t_{32}t_{11} )\,x_{31}x_{12}             \\
	         & \qquad
	+( t_{21}t_{13}-t_{23}t_{11})\,x_{21}x_{13}
	+( t_{31}t_{13}-t_{33}t_{11})\,x_{31}x_{13}
	+( t_{32}t_{13}-t_{33}t_{12})\,x_{32}x_{13}             \\
	         & \qquad
	+(t_{32}t_{21} - t_{31}t_{22})\,x_{32}x_{21}
	+(t_{31}t_{23}-t_{33}t_{21} )\,x_{31}x_{23}
	+( t_{32}t_{23}-t_{33}t_{22})\,x_{32}x_{23},            \\
	\sigma_3 & =
	((t_{33}t_{22} - t_{32}t_{23})\,t_{11}
	+( t_{31}t_{23}-t_{33}t_{21})\,t_{12}
	+(t_{32}t_{21} - t_{31}t_{22})\,t_{13})(x_{31}x_{23}x_{12} +x_{32}x_{21}x_{13}).
\end{align*}
Let
\[
	G^{(3)}(X, T) = \sum_\nu P_\nu(T) X^\nu
\]
be as defined in \eqref{eq:G^3}. We define
\[
	P_\nu^D = \frac{1}{(\kappa)^{(n)} (\kappa - 1)^{(n)}} P_\nu,
\]
where $n = |\nu|$, and \((x)^{(r)} = x(x+1)\cdots(x+r-1)\) denotes the ascending Pochhammer symbol.
Then the collection $\{P_\nu^D\}$ forms the descending basis.
To simplify notation, we write $P_{(\nu_{12},\nu_{13},\nu_{23},\nu_{21},\nu_{31},\nu_{32})}=P_\nu$ for
$\nu=	\left(
	\begin{smallmatrix}0&\nu_{12}&\nu_{13}\\\nu_{21}&0&\nu_{23}\\\nu_{31}&\nu_{32}&0\end{smallmatrix}\right)$.
There, only one representative is listed by symmetry.
\begin{align*}
	|\nu|=0:\hspace{10mm} &                                                                                                            \\
	P_{(0,0,0,0,0,0)}     & =1                                                                                                         \\
	|\nu|=1:\hspace{10mm} &                                                                                                            \\
	P_{(1,0,0,0,0,0)}     & =(\kappa-1)t_{12}                                                                                          \\
	|\nu|=2:\hspace{10mm} &                                                                                                            \\
	P_{(2,0,0,0,0,0)}     & =\frac{(\kappa-1)^{(2)}}{2}t_{12}^2                                                                        \\
	P_{(1,1,0,0,0,0)}     & =(\kappa-1)^{(2)}t_{12}t_{13}                                                                              \\
	P_{(1,0,1,0,0,0)}     & =\kappa^2 t_{12}t_{23}-\kappa t_{22}t_{13}                                                                 \\
	P_{(1,0,0,1,0,0)}     & =\kappa^2 t_{12}t_{21}-\kappa t_{22}t_{11}                                                                 \\
	|\nu|=3:\hspace{10mm} &                                                                                                            \\
	P_{(3,0,0,0,0,0)}     & =\frac{(\kappa-1)^{(3)}}{6}t_{12}^3                                                                        \\
	P_{(2,1,0,0,0,0)}     & =\frac{(\kappa-1)^{(3)}}{2}t_{12}^2t_{13}                                                                  \\
	P_{(2,0,1,0,0,0)}     & =\frac{\kappa(\kappa+1)^2}{2}t_{12}^2t_{23}-\kappa(\kappa+1)t_{22}t_{12}t_{13}                             \\
	P_{(1,1,1,0,0,0)}     & =\kappa^2(\kappa+1)t_{12}t_{13}t_{23}-\kappa(\kappa+1)t_{22}t_{13}^2 \hspace{75mm}                         \\\qquad
	P_{(1,0,1,0,1,0)}     & =\frac{\kappa(\kappa+1)(\kappa^2-2)}{\kappa - 2}t_{12}t_{31}t_{23}
	+\frac{2\kappa(\kappa+1)}{\kappa - 2}(t_{11}t_{22}t_{33}+t_{13}t_{21}t_{32})                                                       \\
	                      & \qquad-\frac{\kappa^2(\kappa+1)}{\kappa - 2}(t_{11}t_{23}t_{32}+t_{22}t_{13}t_{31}+t_{33}t_{12}t_{21})     \\[3pt]
	|\nu|=4:\hspace{10mm} &                                                                                                            \\
	P_{(4,0,0,0,0,0)}     & =\frac{(\kappa-1)^{(4)}}{24}t_{12}^4                                                                       \\[3pt]
	P_{(3,1,0,0,0,0)}     & =\frac{(\kappa-1)^{(4)}}{6}t_{12}^3t_{13}                                                                  \\[3pt]
	P_{(3,0,1,0,0,0)}     & =\frac{\kappa(\kappa+1)(\kappa+2)^2}{6}t_{12}^3t_{23}
	-\frac{\kappa(\kappa+1)(\kappa+2)}{2}t_{22}t_{12}^2t_{13}                                                                          \\[3pt]
	P_{(2,2,0,0,0,0)}     & =\frac{(\kappa-1)^{(4)}}{4}t_{12}^2t_{13}^2                                                                \\[3pt]
	P_{(2,0,2,0,0,0)}     & =\frac{(\kappa+1)^2(\kappa+2)^2}{4}t_{12}^2t_{23}^2                                                        \\[3pt]
	                      & \quad+\frac{(\kappa+1)(\kappa+2)}{2}t_{22}^2t_{13}^2-(\kappa+1)^2(\kappa+2)t_{22}t_{12}t_{13}t_{23}        \\[3pt]
	P_{(2,1,1,0,0,0)}     & =\frac{\kappa(\kappa+1)^2(\kappa+2)}{2}t_{12}^2t_{13}t_{23}
	-\kappa(\kappa+1)(\kappa+2)t_{22}t_{12}t_{13}^2                                                                                    \\[3pt]
	P_{(2,1,0,1,0,0)}     & =\frac{\kappa(\kappa+1)(\kappa+2)^2}{2}t_{12}^2t_{13}t_{21}
	-\frac{\kappa(\kappa+1)(\kappa+2)}{2}t_{11}t_{12}^2t_{23}                                                                          \\[3pt]
	                      & \quad-\kappa(\kappa+1)(\kappa+2)t_{11}t_{22}t_{12}t_{13}                                                   \\[3pt]
	P_{(2,0,1,0,1,0)}     & =\frac{(\kappa+1)(\kappa+2)(2\kappa+1)}{\kappa-2}(t_{11}t_{22}t_{33}t_{12}+t_{12}t_{13}t_{21}t_{32})       \\[3pt]
	                      & \quad+(\kappa+1)(\kappa+2)t_{11}t_{22}t_{12}t_{13}                                                         \\
	                      & \quad
	+\frac{(\kappa+1)(\kappa+2)(\kappa^3+2\kappa^2-2\kappa-2)}{2(\kappa-2)}t_{12}^2t_{23}t_{31}                                        \\[3pt]
	                      & \quad-\frac{(\kappa+1)(\kappa+2)(\kappa^2+2\kappa+2)}{2(\kappa-2)}t_{33}t_{12}^2t_{21}                     \\
	                      & \quad
	-\frac{(\kappa+1)(\kappa+2)(\kappa^2+\kappa-1)}{\kappa-2}(t_{11}t_{12}t_{23}t_{32}+t_{22}t_{12}t_{13}t_{31})                       \\[3pt]
	P_{(1,1,1,1,0,0)}     & =(\kappa+1)^3(\kappa+2)t_{12}t_{13}t_{21}t_{23}
	-(\kappa+1)^2(\kappa+2)(t_{11}t_{12}t_{23}^2+t_{22}t_{13}^2t_{21})                                                                 \\
	                      & \quad
	-(\kappa-1)(\kappa+1)(\kappa+2)t_{11}t_{22}t_{13}t_{23}                                                                            \\[3pt]
	P_{(1,0,1,1,1,0)}     & =\frac{(\kappa+1)(\kappa+2)(3\kappa-1)}{\kappa-2}t_{11}t_{22}t_{33}t_{21}                                  \\[3pt]
	                      & \quad	+\frac{(\kappa+1)(\kappa+2)(\kappa^3+\kappa^2-3\kappa-1)}{2(\kappa-2)}t_{12}t_{21}t_{23}t_{31}       \\
	                      & \quad
	+\frac{(\kappa+1)(\kappa+2)(2\kappa+1)}{\kappa-2}t_{13}t_{21}^2t_{32}                                                              \\[3pt]
	                      & \quad-\frac{(\kappa+1)(\kappa+2)(\kappa^2+1)}{\kappa-2}(t_{11}t_{21}t_{23}t_{32}+t_{22}t_{13}t_{21}t_{31}) \\
	                      & \quad
	-\frac{(\kappa+1)(\kappa+2)(\kappa^2+\kappa-1)}{\kappa-2}t_{33}t_{12}t_{21}^2
	-(\kappa+1)^2(\kappa+2)t_{11}t_{22}t_{23}t_{31}                                                                                    \\
\end{align*}
\bibliography{hermdiffop}
\bibliographystyle{plain}

\end{document}